\documentclass[reqno,11pt]{amsart}

\usepackage{amsmath,amsfonts,amssymb,amsthm,epsfig,amscd,}
\usepackage[new]{old-arrows}
\usepackage[]{fontenc}
\usepackage[latin1]{inputenc}
\usepackage{enumerate}
\usepackage[cmtip,all]{xy}
\usepackage{tabularx,multirow}
\usepackage{relsize}
\usepackage{layout}
\usepackage{fullpage}
\usepackage{color}
\usepackage{tikz-cd}
 \allowdisplaybreaks 
\numberwithin{equation}{section}
\newtheorem{theorem}{Theorem}[section]

\newtheorem{corollary}[theorem]{Corollary}
\newtheorem{lemma}[theorem]{Lemma}

\newtheorem{claim}[theorem]{Claim}

\theoremstyle{definition}
\newtheorem{definition}[theorem]{Definition}

\newtheorem{remark}[theorem]{Remark}

\DeclareMathOperator{\SP}{SP}

\DeclareMathOperator{\Span}{Span}
\DeclareMathOperator{\Div}{Div}
\DeclareMathOperator{\Supp}{Supp}
\DeclareMathOperator{\gon}{gon}
\DeclareMathOperator{\covgon}{cov.gon}
\DeclareMathOperator{\conngon}{conn.gon}
\DeclareMathOperator{\mult}{mult}
\DeclareMathOperator{\Pic}{Pic}

\DeclareMathOperator{\Hilb}{Hilb}

\def\bP{{\mathbb P}}
\def\cE{{\mathcal E}}

\begin{document}

\title[Moving curves of least gonality on symmetric products of curves]{Moving curves of least gonality on symmetric products of curves}

\author{Francesco Bastianelli}
\address{Francesco Bastianelli, Dipartimento di Matematica, Universit\`{a} degli Studi di Bari Aldo Moro, Via Edoardo Orabona 4, 70125 Bari -- Italy}
\email{francesco.bastianelli@uniba.it}

\author{Nicola Picoco}
\address{Nicola Picoco, Dipartimento di Matematica, Universit\`{a} degli Studi di Bari Aldo Moro, Via Edoardo Orabona 4, 70125 Bari -- Italy}
\email{nicola.picoco@gmail.com}

\begin{abstract}
Let $C$ be a smooth complex projective curve of genus $g$ and let $C^{(k)}$ be its $k$-fold symmetric product.
The covering gonality of $C^{(k)}$ is the least gonality of an irreducible curve $E\subset C^{(k)}$ passing through a general point of $C^{(k)}$.
It follows from previous works of the authors that if $2\leq k\leq 4$ and $g\geq k+4$, the covering gonality of $C^{(k)}$ equals the gonality of $C$.
In this paper, we prove that under mild assumptions of generality on $C$, the only curves $E\subset C^{(k)}$ computing the covering gonality of $C^{(k)}$ are copies of $C$ of the form $C+p$, for some point $p\in C^{(k-1)}$. 
As a byproduct, we deduce that the connecting gonality of $C^{(k)}$---i.e. the least gonality of an irreducible curve $E\subset C^{(k)}$ connecting two general points of $C^{(k)}$---is strictly larger than the covering gonality.
\end{abstract}

\thanks{The authors are members of INdAM (GNSAGA)} 

\maketitle

\section{Introduction}\label{section:intro}

In this paper, which is a sequel of \cite{BP1}, we study families of irreducible curves covering the $k$-fold symmetric product $C^{(k)}$ of a smooth complex projective curve $C$.
In particular, we are aimed at characterizing families of curves which compute the covering gonality of $C^{(k)}$.

\smallskip
We recall that the \emph{gonality} $\gon(E)$ of an irreducible complex projective curve $E$ is the least integer $\delta$ such that there exists a non-constant morphism $\widetilde{E}\longrightarrow \mathbb{P}^1$ of degree $\delta$, where $\widetilde{E}$ is the normalization of $E$.
Given an irreducible complex projective variety $X$, the \emph{covering gonality} of $X$ is the positive integer
\begin{displaymath}
\covgon(X):=\min\left\{d\in \mathbb{N}\left|
\begin{array}{l}
\text{Given a general point }x\in X,\,\exists\text{ an irreducible}\\ \text{curve } E\subseteq X  \text{ such that }x\in E \text{ and }\gon(E)=d
\end{array}\right.\right\},
\end{displaymath}
i.e. it is the least integer $d$ such that there exists a family $\mathcal{E}\stackrel{\pi}{\longrightarrow}T$ of irreducible curves covering an open subset of $X$, where for general $t\in T$, the curve $E_t:=\pi^{-1}(t)\subset X$ has gonality $\gon(E_t)=d$. 

The covering gonality is an important birational invariant which, together with other measures of irrationality, has been widely studied in the last decade (see e.g. \cite{B1,BDELU,BCFS,GK,Mar,V}).
In particular, as $\covgon(X)=1$ if and only if $X$ is covered by rational curves, we can think of the covering gonality of $X$ as a measure of the failure of $X$ to be uniruled.

A natural problem concerning the covering gonality is classifying families of curves $\mathcal{E}\stackrel{\pi}{\longrightarrow}T$ as above which compute the covering gonality of the variety $X$.
In this direction, \cite{LP} investigates the covering gonality of smooth surfaces $S\subset \mathbb{P}^3$ of degree $d\geq 5$, proving that $\covgon(S)=d-2$ and characterizing families of irreducible curves on $S$ whose general member has gonality exactly $d-2$ (cf. \cite[Corollaries 1.7 and 1.8]{LP}).
The importance of studying the gonality of moving curves is also highlighted by its applications, as e.g. to the study of Seshadri constants and the ample cone on projective surfaces (cf. \cite{B2,KSS}). 

\smallskip
In the following, we are interested in classifying families of curves computing $\covgon(X)$, when $X$ is the $k$-fold symmetric products of a smooth complex projective curve $C$.
Given a positive integer $k$, we recall that the $k$-fold symmetric product of $C$ is the smooth projective variety $C^{(k)}$ parameterizing effective divisors of degree $k$ on $C$, so that the points $P\in C^{(k)}$ are unordered $k$-tuples $p_1+\dots+p_k$ of points $p_j\in C$.

As far as the covering gonality is concerned, it follows from \cite[Theorem 1.6]{B1} and \cite[Theorem 1.1]{BP1} that if $2\leq k\leq 4$ and $C$ has genus $g\geq k+4$, then 
\begin{equation}\label{eq:covgon}
\covgon\big(C^{(k)}\big)=\gon(C).
\end{equation}
It is also worth noticing that $C^{(k)}$ is covered by copies of $C$ of the form 
\begin{equation}\label{eq:C_Q}
C_Q:=\left\{\left.Q+p\in C^{(k)}\right|p\in C\right\},
\end{equation} 
where $Q=p_1+\dots+p_{k-1}\in C^{(k-1)}$ is a fixed point. 
In particular, the family $\mathcal{C}\stackrel{\pi}{\longrightarrow} C^{(k-1)}$ with $C_Q=\pi^{-1}(Q)$ computes the covering gonality of $C^{(k)}$.

\smallskip
We prove that if $C$ is sufficiently general and $2\leq k\leq 4$, then the curves $C_Q\subset C^{(k)}$ defined in \eqref{eq:C_Q} are the only irreducible curves covering $C^{(k)}$ and having the same gonality as $C$.

\begin{theorem}\label{thm:covfam}
Let $k\in\{2,3,4\}$ and let $C$ be a smooth complex projective curve of genus $g\geq k+4$.
Suppose that $C$ is not hyperelliptic, bielliptic, or a smooth plane quintic.\\
Given a general point $x\in C^{(k)}$, if $E\subset C^{(k)}$ is an irreducible curve passing through $x$ such that $\gon(E)=\gon(C)$, then $E=C_Q:=\left\{\left.Q+p\in C^{(k)}\right|p\in C\right\}$ for some point $Q\in C^{(k-1)}$.\\
Equivalently, if $\mathcal{E}\stackrel{\pi}{\longrightarrow} T$ is a family of irreducible curves covering an open subset of $C^{(k)}$ such that for general $t\in T$, the fiber $E_t:=\pi^{-1}(t)\subset C^{(k)}$ is an irreducible curve with $\gon(E_t)=\gon(C)$, then the family is isotrivial with $E_t=C_Q$ for some point $Q\in C^{(k-1)}$.
\end{theorem}

We point out that when $k=2$, the assumptions on $C$ can be slightly weakened (cf. Remark \ref{rem:plane quintic}), so that Theorem \ref{thm:covfam} does improve \cite[Theorem 1.5]{B1}. 

\smallskip
Given an irreducible complex projective variety $X$, another relevant measure of irrationality introduced in \cite{BDELU} is the \emph{connecting gonality} of $X$, which is defined as the positive integer
\begin{displaymath}
\conngon(X):=\min\left\{d\in \mathbb{N}\left|
\begin{array}{l}
\text{Given two general points }x,y\in X,\,\exists\text{ an irreducible}\\ \text{curve } E\subseteq X  \text{ such that }x,y\in E \text{ and }\gon(E)=d
\end{array}\right.\right\}.
\end{displaymath}
In particular, $\conngon(X)=1$ if and only if $X$ is rationally connected.
Of course, 
$$\conngon(X)\geq \covgon(X),$$ 
and equality may hold (see e.g. \cite[Theorem 1.3]{B3}).
Moreover, it easy to see that if $\mathcal{E}\stackrel{\pi}{\longrightarrow}T$ is a family of irreducible curves on $X$ computing the connecting gonality of $X$, then $\dim T\geq 2\dim X-2$  (cf. \cite[Remark 4.2]{BCFS2}).

It follows that if $X=C^{(k)}$, the family $\mathcal{C}\stackrel{\pi}{\longrightarrow} C^{(k-1)}$ such that $\pi^{-1}(Q)=C_Q$ is too small for computing the connecting gonality of $C^{(k)}$.
Thus Theorem \ref{thm:covfam} gives the following.    

\begin{corollary}\label{cor:conngon}
Let $k\in\{2,3,4\}$ and let $C$ be a smooth complex projective curve of genus $g\geq k+4$.
Suppose that $C$ is not hyperelliptic, bielliptic, or a smooth plane quintic.
Then
\begin{equation}\label{eq:conngon}
\conngon\big(C^{(k)}\big)>\gon(C)
\end{equation}
and, in particular, $\conngon\big(C^{(k)}\big)>\covgon\big(C^{(k)}\big)$. 
\end{corollary}

\smallskip
The proof of Theorem \ref{thm:covfam} relies on the very same argument as \cite[Theorem 1.1]{BP1}.
In particular, we consider the canonical embedding $C\longhookrightarrow \mathbb{P}^{g-1}$ and for a general point $P=p_1+\dots+p_{k}\in C^{(k)}$, we define the $(k-1)$-plane $\Lambda_P:=\Span\left(p_1,\dots,p_k\right)\subset \mathbb{P}^{g-1}$.
Let $E\subset C^{(k)}$ be a $d$-gonal curve passing through a general point of $C^{(k)}$ with $d=\gon(C)$, and let $P_1,\dots,P_d\in E$ be the points corresponding to a general fiber of the $d$-gonal map $\widetilde{E}\longrightarrow \mathbb{P}^1$.
Then it turns out that the linear span $\Sigma$ of the $(k-1)$-planes $\Lambda_{P_1},\ldots,\Lambda_{P_d}\subset \mathbb{P}^{g-1}$ has dimension much smaller than expected (cf. Theorem \ref{thm:dimSP}). 
Since $\Sigma$ is actually the linear span of the points of $C$ supporting the points $P_1,\dots,P_d\in C^{(k)}$, it cuts out on $C$ a linear series of dimension $\deg (\Sigma\cap C)-1-\dim \Sigma$.
Then it takes a great deal of case work to decide when such a linear series exists, and we eventually conclude that there exists $Q\in C^{(k-1)}$ such that the points $P_1,\dots,P_d\in E$ corresponding to a general fiber of the $d$-gonal map lie on the curve $C_Q$.

\smallskip
The paper is organized as follows.
In Section \ref{section:prel}, we introduce the main techniques involved in the proof of Theorem \ref{thm:covfam}.
In particular, we collect the main results of \cite[Section 2]{BP1}, which concern the dimension of the linear span of linear subspaces of $\mathbb{P}^{n}$ satisfying a condition of Cayley--Bacharach type, and we present the framework of \cite{B1} connecting these linear spaces to families of curves covering $C^{(k)}$. 
Moreover, we briefly recall some standard results concerning Brill-Noether theory.

Finally, Section \ref{section:proof} is devoted to prove Theorem \ref{thm:covfam} and Corollary \ref{cor:conngon}.

\subsection*{Notation}

We work throughout over the field $\mathbb{C}$ of complex numbers.
When we speak of a \emph{smooth} curve, we always implicitly assume it to be irreducible.
We say that a property holds for a \emph{general} (resp. \emph{very general}) point ${x\in X}$ if it holds on a Zariski open nonempty subset of $X$ (resp. on the complement of the countable union of proper subvarieties of $X$).


\medskip
\section{Preliminaries}\label{section:prel} 

In this section, we present the main results involved in the proof of Theorem \ref{thm:covfam}.

\smallskip\subsection{Linear subspaces of $\mathbb{P}^n$ in special position}\label{sub:sp} 

Let us fix integers $1\leq k\leq n$ and $d\geq 2$.  
In this section, we review some notions and results included in \cite[Section 2]{BP1}.
In particular, we are interested in collections of $(k-1)$-dimensional linear spaces in $\mathbb{P}^n$ satisfying a property of Cayley-Bacharach type, defined as follows (cf. \cite[Definition 3.2]{B1}).

\begin{definition}\label{def:SP}
Let $\Lambda_1,\dots,\Lambda_d\subset \mathbb{P}^{n}$ be linear subspaces of dimension $k-1$.
We say that $\Lambda_1,\dots,\Lambda_d$ are \emph{in special position with respect to} $(n-k)$\emph{-planes}, or just $\SP(n-k)$, if for any $j=1,\dots,d$ and for any $(n-k)$-plane $L\subset \mathbb{P}^n$ intersecting $\Lambda_1,\dots,\widehat{\Lambda_j},\dots,\Lambda_d$, we have that $L$ meets $\Lambda_j$ too.
\end{definition}

\begin{remark}\label{rem:unionSP} 
We note that the linear spaces $\Lambda_1,\dots,\Lambda_d$ are not necessarily distinct.
In particular, two $(k-1)$-planes $\Lambda_1,\Lambda_2$ are $\SP(n-k)$ if and only if they coincide.
\end{remark}

\begin{definition}\label{def:partition}
Let $\Lambda_1,\dots,\Lambda_d\subset \mathbb{P}^{n}$ be linear subspaces of dimension $k-1$ in special position with respect to $(n-k)$-planes.
We say that the sequence $\Lambda_1,\dots,\Lambda_d$ is \emph{decomposable} if there exists a non-trivial partition of the set of indexes $\{1,\dots,d\}$, where each part $\{i_1,\dots,i_t\}$ is such that the corresponding $(k-1)$-planes $\Lambda_{i_1},\dots,\Lambda_{i_t}$ are $\SP(n-k)$.
Otherwise, we say that the sequence is \emph{indecomposable}.
\end{definition}

The following result collects \cite[Theorem 2.5]{BP1} and \cite[Corollary 2.6]{BP1}, which bound the dimension of the linear span of a sequence of $(k-1)$-planes satisfying property $\SP(n-k)$.

\begin{theorem}\label{thm:dimSP}
Let $\Lambda_1,\dots,\Lambda_d\subset \mathbb{P}^{n}$ be linear subspaces of dimension $k-1$ in special position with respect to $(n-k)$-planes.
Consider a partition of the set of indexes $\{1,\dots,d\}$ such that each part $\{i_1,\dots,i_t\}$ corresponds to an indecomposable sequence $\Lambda_{i_1},\dots,\Lambda_{i_t}$ of $(k-1)$-planes that satisfy $\SP(n-k)$.
Denoting by $m$ the number of parts of the partition, we have that
$$
\dim\Span\left(\Lambda_1,\dots,\Lambda_d\right)\leq d+k-3+(m-1)(k-2).
$$
In particular, if the sequence of $(k-1)$-planes $\Lambda_1,\dots,\Lambda_d\subset \mathbb{P}^{n}$ is indecomposable, then $m=1$ and 
$$
\dim\Span\left(\Lambda_1,\dots,\Lambda_d\right)\leq d+k-3.
$$
\end{theorem}

Finally, we recall the behavior of property $\SP(n-k)$ under projection from a linear subspace of $\mathbb{P}^n$.
To this aim, let us consider a linear subspace $L\subset \mathbb{P}^n$ of dimension $\alpha:=\dim L\geq 0$, and let $\pi_L\colon\mathbb{P}^n \dashrightarrow \mathbb{P}^{n-\alpha-1}$ denote the projection from $L$. 
Therefore, if $\Lambda\subset \mathbb{P}^n$ is a linear subspace with $\dim(L\cap \Lambda)=\beta$, then $\dim \pi_L(\Lambda)=\dim \Lambda-\beta-1$. 
Besides, for any linear subspace $\Gamma\subset \mathbb{P}^{n-\alpha-1}$, the Zariski closure of $\pi_L^{-1}(\Gamma)\subset \mathbb{P}^n$ is a linear subspace of dimension $\dim \Gamma+\alpha+1$ containing $L$. 

Using this notation, the following holds (cf. \cite[Lemma 2.13]{BP1}).

\begin{lemma}\label{lem:sottoSP}
Let $\Lambda_1,\dots,\Lambda_d\subset\mathbb{P}^n$ be $(k-1)$-planes that satisfy $\SP(n-k)$, and let $L\subset \mathbb{P}^n$ be a linear subspace of dimension $\alpha\geq 0$. 
Let $2\leq r\leq d-1$ be an integer such that $L\cap \Lambda_i=\emptyset$ for any $i=1,\dots,r$, and $L\cap \Lambda_i\neq \emptyset$ for any $i=r+1,\dots,d$.
Consider the $(k-1)$-planes $\Gamma_i:=\pi_L(\Lambda_i)\subset \mathbb{P}^{n-\alpha-1}$, with $i=1,\dots,r$. 
Then $\Gamma_1,\dots,\Gamma_r$ are $\SP(n-\alpha-1-k)$.
\end{lemma}

\subsection{Covering families on symmetric products of curves}\label{sub:covgon}
In this section, we are aimed at discussing the covering gonality in terms of covering families.
Moreover, we briefly recall the framework of \cite{B1} in order to explain the relation between covering families on $C^{(k)}$ and linear subspaces of $\mathbb{P}^{g-1}$ having dimension $k-1$ and satisfying property $\SP(g-1-k)$.

\smallskip
Let us consider an irreducible complex projective variety $X$ of dimension $n$.

\begin{definition}\label{def:covfam}
A \emph{covering family of} $d$\emph{-gonal curves} on $X$ consists of a smooth family $\mathcal{E}\stackrel{\pi}{\longrightarrow} T$ of irreducible curves, together with a dominant morphism $f\colon \mathcal{E}\longrightarrow X$ such that for general $t\in T$, the fiber $E_t:=\pi^{-1}(t)$ is a smooth curve with gonality $\gon(E_t)=d$ and the restriction $f_{t}\colon E_t\longrightarrow X$ of $f$ is birational onto its image. 
\end{definition}

In these terms, the covering gonality of $X$ coincides with the least integer $d$ for which a covering family of $d$-gonal curves on $X$ exists  (cf. \cite[Lemma 2.1]{GK}), and the irreducible curves on $X$ computing the covering gonality are the images $f(E_t)$.

\begin{remark}\label{rem:covfam}
Let $\mathcal{E}\stackrel{\pi}{\longrightarrow} T$ be a family of irreducible curves covering an open subset of $X$ such that for general $t\in T$, the fiber $E_t:=\pi^{-1}(t)\subset X$ is an irreducible curve with $\gon(E_t)=d$.
By making a base change 
\begin{equation*}
\xymatrix{ \mathcal{F} \ar[r]^\nu \ar[d]_{\psi} & \mathcal{E}\ar[d]^{\pi} \\  U \ar[r]^{h} & T}
\end{equation*}
and possibly shrinking $U$, we obtain a smooth family $\mathcal{F}\stackrel{\psi}{\longrightarrow} U$ such that, given a general $t\in T$ and $u\in h^{-1}(t)$, the restriction $\nu_{|F_u}\colon F_u\longrightarrow E_t$ is the normalization map.
Thus $\mathcal{F}\stackrel{\psi}{\longrightarrow} U$ is a covering family of $d$-gonal curves on $X$.
\end{remark}

\smallskip
Let $C$ be a smooth projective curve of genus $g\geq 3$. 
For $2\leq k\leq g-1$, let $C^{(k)}$ be its $k$-fold symmetric product and let $\mathcal{E}\stackrel{\pi}{\longrightarrow} T$ be a covering family of $d$-gonal curves on $C^{(k)}$.
According to \cite[Remark 1.5]{BDELU}, we may assume that both the varieties $T$ and $\mathcal{E}$ are smooth, with $\dim(T)=k-1$.
Furthermore, up to base changing $T$, there is a commutative diagram
\begin{equation}\label{diagram:covgon}
\xymatrix{C^{(k)}  & \mathcal{E} \ar[l]_-{f} \ar[dr]_-{\pi} \ar[r]^-\varphi & T\times \mathbb{P}^1 \ar[d]^-{\mathrm{pr_1}} \\ & & T,\\}
\end{equation}
where the restriction $\varphi_t\colon E_t \longrightarrow \{t\}\times\mathbb{P}^1  \cong \mathbb{P}^1$ is a $d$-gonal map (cf. \cite[Example 4.7]{B1}). 

\smallskip
Let ${\phi\colon C\longrightarrow \mathbb{P}^{g-1}}$ be the canonical map of $C$, and let $\mathbb{G}(k-1,g-1)$ be the Grassmannian parameterizing ${(k-1)}$-planes in
$\mathbb{P}^{g-1}$.
Then we define the rational map 
\begin{equation}\label{diagram:gauss}
\gamma\colon C^{(k)}\dashrightarrow \mathbb{G}(k-1,g-1),
\end{equation} 
sending a general point ${p_1+\dots+p_k\in C^{(k)}}$ to the point of $\mathbb{G}(k-1,g-1)$ which parameterizes the $(k-1)$-plane $\Span\left(\phi(p_1),\dots,\phi(p_k)\right)\subset \mathbb{P}^{g-1}$. 

So we can state the following result, which is somehow the core of the proof of Theorem \ref{thm:covfam}. 
In particular, it expresses the relation between covering families of irreducible $d$-gonal curves on $C^{(k)}$ and $(k-1)$-planes of $\mathbb{P}^{g-1}$ in special position with respect to $(g-1-k)$-planes (cf. \cite[Section 4]{B1} and \cite[Theorem 3.2]{BP1}).

\begin{theorem}\label{thm:covfamCk}
Let $C$ be a smooth projective curve of genus $g\geq 2$ and let $C^{(k)}$ be its $k$-fold symmetric product, with $2\leq k\leq g-1$.
Let $\mathcal{E}\stackrel{\pi}{\longrightarrow} T$ be a covering family of $d$-gonal curves on $C^{(k)}$.
Using notation as above, consider a general point $(t,y)\in T\times \mathbb{P}^1$ and let $\varphi^{-1}(t,y)=\left\{x_1,\ldots,x_d\right\}\subset E_t$ be its fiber.\\
Then the $(k-1)$-planes parameterized by $(\gamma\circ f) (x_1),\dots,(\gamma\circ f) (x_d)\in \mathbb{G}(k-1,g-1)$ are in special position with respect to $(g-1-k)$-planes of $\mathbb{P}^{g-1}$.
\end{theorem}

\subsection{Brill-Noether theory}\label{sub:bn}
In this section we recall some results regarding the existence of linear series on a smooth projective curve $C$.

Given two positive integers $r$ and $d$, let $W^r_d$ be the subvariety of $\Pic^d(C)$, which parameterizes complete linear series on $C$ having degree $d$ and dimension at least $r$ (cf. \cite[p.\,153]{ACGH}). 
In the proof of Theorem \ref{thm:covfam}, we frequently use Martens' theorem and its refinement due to Mumford, concerning the dimension of $\dim W^r_d$ (see \cite[Theorems IV.5.1 and IV.5.2]{ACGH}).
So we summarize these results in the following theorem. 

\begin{theorem}\label{thm:mm}
Let $C$ be a smooth non-hyperelliptic curve of genus $g\geq 4$. 
Let $d$ and $r$ be two integers such that $2\leq d\leq g-1$ and $0< 2r\leq d$. 
Then
\begin{equation}
\label{eq:martens}
\dim W^r_d\leq d-2r-1.
\end{equation}
If in addition $d\leq g-2$ and $C$ is not trigonal, bielliptic, or a smooth plane quintic, then 
\begin{equation}
\label{eq:mumford}
\dim W^r_d\leq d-2r-2.
\end{equation}
\end{theorem}

Finally, we recall the following useful lemma (see \cite[Lemma 3.5]{BP1}).

\begin{lemma}\label{lem:very ample}
Let $C$ be a smooth non-hyperelliptic curve of genus $g\geq 3$ and, for some effective divisor $D\in \Div(C)$, let $|D|$ be a $\mathfrak{g}^r_d$ such that $r\geq 2$ and $\dim|D-R|= 0$ for any $R\in C^{(r)}$.
Then $|D|$ is a very ample $\mathfrak{g}^{2}_{d}$ on $C$.
\end{lemma}

\medskip
\section{Proof of Theorem \ref{thm:covfam}}\label{section:proof}
In this section, we are aimed at proving Theorem \ref{thm:covfam} and Corollary \ref{cor:conngon}.
In the light of Remark \ref{rem:covfam}, the assertion of Theorem \ref{thm:covfam} can be rephrased in terms of covering families on $C^{(k)}$, as follows.

\begin{theorem}\label{thm:covfam2}
Let $k\in\{2,3,4\}$ and let $C$ be a smooth complex projective curve of genus $g\geq k+4$.
Suppose that $C$ is not hyperelliptic, biellliptic, or a smooth plane quintic.\\
Let $\mathcal{E}\stackrel{\pi}{\longrightarrow} T$ be a smooth family of curves, endowed with a dominant morphism $f\colon \mathcal{E}\longrightarrow C^{(k)}$ such that for general $t\in T$, the fiber $E_t:=\pi^{-1}(t)$ is an irreducible curve with $\gon(E_t)=\gon(C)$ and the restriction $f|_{E_t}\colon E_t\longrightarrow C^{(k)}$ is birational onto its image.
Then  
$$E_t\cong f(E_t)=C_Q:=\left\{\left.Q+p\in C^{(k)}\right|p\in C\right\}$$ 
for some point $Q\in C^{(k-1)}.$
\end{theorem}

Therefore, in order to achieve Theorem \ref{thm:covfam}, we prove Theorem \ref{thm:covfam2}.

\begin{proof}[Proof of Theorem \ref{thm:covfam2}]
Since $C$ is non-hyperelliptic, we can identify $C$ with its canonical model in $\mathbb{P}^{g-1}$, in order to simplify the notation.

Let $\mathcal{E}\stackrel{\pi}{\longrightarrow} T$ be a covering family of $d$-gonal curves on $C^{(k)}$ as above, with $d=\gon(C)\geq 3$. 
Using notation as in \eqref{diagram:covgon}, we consider the $d$-gonal map $\varphi_t\colon E_t \longrightarrow \{t\}\times\mathbb{P}^1$, and for a general point $(t,y)\in T\times \mathbb{P}^1$, let $\varphi_t^{-1}(t,y)=\left\{x_1,\ldots,x_d\right\}\subset E_t$ be its fiber.
Then, for any $i=1,\dots,d$, we define $P_i:=f_t(x_i)\in C^{(k)}$ and we set
\begin{equation}\label{eq:P_i}
P_1=p_1+\dots+p_{k} \quad P_2=p_{k+1}+\dots+p_{2k} \quad \dots\, \quad P_d=p_{(d-1)k+1}+\dots+p_{dk}.
\end{equation}
It follows from the generality of $(t,y)\in T\times \mathbb{P}^1$ that the points $P_i\in C^{(k)}$ are distinct, and they lie outside the diagonal of $C^{(k)}$.
Furthermore, for each $i=1,\dots,d$, we consider the linear space spanned by $\Supp (P_i)$,
\begin{equation}\label{eq:Lambda_i}
\Lambda_i:=\Span\left(p_{(i-1)k+1},\dots,p_{ik}\right)\subset \mathbb{P}^{g-1},
\end{equation}
which has dimension $k-1$ and is parameterized by the point $\gamma(P_i)\in \mathbb{G}(k-1,g-1)$, where $\gamma$ is the map in \eqref{diagram:gauss}.
Thus Theorem \ref{thm:covfamCk} assures that the $(k-1)$-planes $\Lambda_1,\dots,\Lambda_d$ are in special position with respect to $(g-1-k)$-planes.

Let us define the effective divisor 
\begin{equation}\label{eq:divisorDbis}
D=D_{(t,y)}:=p_1+\dots +p_{dk}=\sum_{j=1}^h n_jq_j \in\Div(C),
\end{equation}
where the points $q_j\in \{p_1,\dots,p_{dk}\}$ are assumed to be distinct and $n_j=\mult_{q_j}(D)$.
We distinguish some cases depending on the values of the integers $n_j$.

\medskip
\underline{Case A:} suppose that $n_j=1$ for any $j=1,\dots,h$, that is the points $p_1,\dots,p_{dk}$ are all distinct.
We consider the sequence $\Lambda_1,\dots,\Lambda_d\subset \mathbb{P}^{g-1}$ of $(k-1)$-planes satisfying $\SP(g-1-k)$, together with a partition of the set of indexes $\left\{1,\dots,d\right\}$ such that each part $\{i_1,\dots,i_t\}$ corresponds to an indecomposable sequence $\Lambda_{i_1},\dots,\Lambda_{i_t}$ of $(k-1)$-planes that satisfy $\SP(g-1-k)$, as in the assumption of Theorem \ref{thm:dimSP}.
\begin{claim}\label{claim:almeno3}
Each part $\{i_1,\dots,i_t\}$ consists of at least $3$ elements. 
\begin{proof}[Proof of Claim \ref{claim:almeno3}]
By contradiction, suppose that there are two $(k-1)$-planes, say $\Lambda_{1}=\Span(p_1,\dots,p_k)$ and $\Lambda_{2}=\Span(p_{k+1},\dots,p_{2k})$, that satisfy $\SP(g-1-k)$.
Then $\Lambda_{1}$ and $\Lambda_{2}$ coincide (cf. Remark \ref{rem:unionSP}), so that $\dim\Span(p_1,\dots,p_{2k})=k-1$.
Setting $P=p_1+\dots+p_{2k}\in \Div(C)$, the geometric version of the Riemann--Roch theorem (see \cite[p.\,12]{ACGH}) yields
\begin{equation}\label{eq:GRR1}
\dim |P|=\deg P-1-\dim\Span(p_1,\dots,p_{2k})=2k-1-(k-1)=k=\frac{\deg D}{2}.
\end{equation}
Since $g\geq k+4$, we have $\deg P=2k\leq 2(g-4)$.
Thus Clifford's theorem (cf. \cite[p.\,107]{ACGH}) assures that either $P=0$, $P$ is a canonical divisor of $C$, or $C$ is hyperelliptic.
Then we get a contradiction, as $0<\deg P<2g-2$ and $C$ is assumed to be non-hyperelliptic.
\end{proof}
\end{claim}

Let $m$ denote the number of parts $\{i_1,\dots,i_t\}$. 
Then $m\leq \left\lfloor\frac{d}{3}\right\rfloor$ by Claim \ref{claim:almeno3}, and Theorem \ref{thm:dimSP} gives
$$
\dim\Span(p_1,\dots,p_{dk})  =\dim\Span\left(\Lambda_1,\dots,\Lambda_d\right)\leq d+k-3 +(m-1)(k-2).
$$
Therefore, the geometric version of the Riemann--Roch theorem implies
\begin{equation}\label{eq:GRR2}
\begin{split}
\dim |D|&= \deg D-1-\dim\Span(p_1,\dots,p_{dk})\geq dk-1-\big(d+k-3 +(m-1)(k-2)\big) \\
& =  (k-1)d -(k-2)m.
\end{split}
\end{equation}

If $k=2$, we obtain that $\dim |D|\geq d=\frac{\deg D}{2}$, so that $\deg D\geq 2g-2$ by Clifford's theorem.
On the other hand, 
\begin{equation}\label{eq:maxgon}
d=\gon(C)\leq \left\lfloor\frac{g+3}{2}\right\rfloor
\end{equation}
(see e.g. \cite[Theorem V.1.1]{ACGH}) and hence $\deg D\leq 2\left\lfloor\frac{g+3}{2}\right\rfloor<2g-2$ for any $g\geq 6$, a contradiction.

If $k=3$, then \eqref{eq:GRR2} yields $\dim |D|\geq 2d -\left\lfloor\frac{d}{3}\right\rfloor>\frac{\deg D}{2}$. 
Thus $\deg D\geq 2g$ by Clifford's theorem and $\dim |D|=3d-g$ by the Riemann--Roch theorem.
Combining this fact and \eqref{eq:GRR2}, we obtain $g\leq d+m\leq d+\left\lfloor\frac{d}{3}\right\rfloor$, which fails for any $g\geq 6$ by \eqref{eq:maxgon}.

Similarly, if $k=4$, we obtain $\dim |D|\geq 3d -2\left\lfloor\frac{d}{3}\right\rfloor>\frac{\deg D}{2}$.
Therefore the Riemann--Roch theorem and \eqref{eq:GRR2} give $\dim |D|=4d-g$ and $\dim |D|\geq 3d-2m$, respectively.
Thus $g\leq d+2m$.
Using $m\leq \left\lfloor\frac{d}{3}\right\rfloor$ and \eqref{eq:maxgon}, it is easy to check that $g\leq d+2m$ if and only if $g\in\{9,10,11,15\}$ with $d=\left\lfloor\frac{g+3}{2}\right\rfloor$ and $m=\left\lfloor\frac{d}{3}\right\rfloor$, i.e. $(g,d,m)\in \{(9,6,2),(10,6,2),(11,7,2),(15,9,3)\}$. 
As each of the $m$ parts contains at least 3 linear spaces among $\Lambda_1,\dots,\Lambda_d$, we deduce that for any triple $(g,d,m)$, there exists a part consisting of exactly 3 linear spaces, say $\Lambda_1,\Lambda_2,\Lambda_3$, which are $\SP(g-5)$.
By Theorem \ref{thm:dimSP}, the linear span of $\Lambda_1,\Lambda_2,\Lambda_3$ has dimension at most 4, and according to \eqref{eq:Lambda_i}, it contains the points $p_1,\dots,p_{12}$.
Using the geometric version of the Riemann--Roch theorem, we obtain that the linear series $|p_1+\dots+p_{12}|$ is a complete $\mathfrak{g}^r_{12}$ with $r\geq 7$.
However, this contradicts Clifford's theorem as $2r>12$, but $12<2g$ for any $g\in\{9,10,11,15\}$.

It follows that Case A does not occur and the points $p_1,\dots,p_{dk}$ are not distinct
 
\medskip
\underline{Case B:} suppose that the points $p_1,\dots,p_{kd}$ are not distinct, that is the integers $n_j$ are not all equal to $1$.
For any integer $\alpha$, let us consider the set  
\begin{equation}\label{eq:N_alpha}
N_{\alpha}:=\left\{\left. q_j\in\Supp(D)\right|n_j=\alpha\right\}, 
\end{equation}
consisting of the points supporting $D$ such that $\mult_{q_j}D=\alpha$. 
For some $\alpha\geq 1$ such that $N_{\alpha}\neq \emptyset$, we may assume that $N_{\alpha}=\{q_1,\dots,q_s\}$, with $s\leq \left\lfloor\frac{\deg D}{\alpha}\right\rfloor=\left\lfloor\frac{kd}{\alpha}\right\rfloor$.

We note that, since $T\times \mathbb{P}^1$ is irreducible, there exists a suitable open subset $U\subset T\times \mathbb{P}^1$ such that, as we vary $(t,y)\in U$, the multiplicities $n_j$ of the points supporting the divisor $D=D_{(t,y)}$ given by \eqref{eq:divisorDbis} do not vary, i.e. the cardinality of each $N_{\alpha}$ is constant on $U$; otherwise, we could distinguish irreducible components of $T\times \mathbb{P}^1$ depending on the configuration of the divisor $D=D_{(t,y)}$.     
Then, we can define a rational map $\xi_{\alpha}\colon T\times \mathbb{P}^1\dashrightarrow C^{(s)}$ which sends a general point $(t,y)\in U$ to the effective divisor $Q:=q_1+\dots+q_s\in C^{(s)}$ such that $N_{\alpha}=\{q_1,\dots,q_s\}$. 

For a general $t\in T$, let us consider the map
\begin{equation*}
\xymatrix{\{t\}\times\bP^1  \ar@{-->}[r]^-{\xi_{\alpha,t}} & C^{(s)} \ar[r]^-{u} & J(C),\\}
\end{equation*}
obtained by composing the restriction $\xi_{\alpha,t}\colon \{t\}\times \mathbb{P}^1\dashrightarrow C^{(s)}$ of $\xi_\alpha$ to $\{t\}\times \mathbb{P}^1$, and the  Abel--Jacobi map $u\colon C^{(s)}\longrightarrow J(C)$.
Since $J(C)$ is an Abelian variety, it does not contain rational curves. 
Thus the map $u\circ \xi_{\alpha,t}$ must be constant.
In the following, we distinguish two cases and, with a great deal of work, we prove that $\xi_{\alpha,t}$ is necessarily constant for some $\alpha$ such that $N_\alpha\neq \emptyset$.

\medskip
\hspace{1cm}\underline{Case B.1:} suppose that for any $\alpha$ such that $N_{\alpha}\neq\emptyset$, the map $\xi_{\alpha,t}\colon\{t\}\times \mathbb{P}^1\dashrightarrow C^{(s)}$ is non-constant, with $N_{\alpha}=\{q_1,\dots,q_s\}$.
Thus $\xi_{\alpha,t}\left(\{t\}\times\bP^1\right)$ is a rational curve on $C^{(s)}$, and Abel's theorem ensures that $|q_1+\dots+q_s|$ is a $\mathfrak{g}^r_s$ containing $\xi_{\alpha,t}\left(\{t\}\times\bP^1\right)$, so that $r\geq 1$ and $s=|N_\alpha|\geq d= \gon(C)$.
Then $\alpha\leq k$, because otherwise we would have $s\leq \left\lfloor\frac{kd}{\alpha}\right\rfloor< d=\gon(C)$.

We claim that $\alpha |N_\alpha|=\alpha s$ is multiple of $d$. 
To see this fact, let $i=1,\dots,d$ and consider the points $P_i=p_{(i-1)k+1}+\dots+p_{ik}$ in \eqref{eq:P_i}. 
We note that the number of points $p_j\in N_\alpha\cap \Supp(P_i)$ does not depend on $i$; otherwise, as we vary $(t,y)\in \{t\}\times \mathbb{P}^1$, the points $P_i$ would describe different irreducible components of $E_t$, but such a curve is irreducible. 
Thus $\alpha|N_\alpha|$---which is the number of points among $p_1,\dots,p_{dk}$ contained in $N_\alpha$---must be a multiple of the number of points $P_i$, which is $d$.

Furthermore, we have that $\deg D=kd=\sum_{\alpha=1}^k \alpha|N_\alpha|$ and there exists $\beta>1$ such that $N_\beta\neq \emptyset$ by assumption of Case B.
By using these facts and working out the combinatorics, it is easy to see that only the following cases may occur:
\begin{itemize}
  \item[(i)] $k=2$, $|N_2|=d$ and $N_1=\emptyset$;
  \item[(ii)] $k=3$, $|N_1|=|N_2|=d$ and $N_3=\emptyset$;
  \item[(iii)] $k=3$, $|N_3|=d$ and $N_1=N_2=\emptyset$;
  \item[(iv)] $k=3$, $|N_2|=\frac{3}{2}d$ and $N_1=N_3=\emptyset$;
  \item[(v)] $k=4$, $|N_1|=|N_3|=d$ and $N_2=N_4=\emptyset$;
  \item[(vi)] $k=4$, $|N_1|=2d$, $|N_2|=d$ and $N_3=N_4=\emptyset$;
  \item[(vii)] $k=4$, $|N_1|=d$, $|N_2|=\frac{3}{2}d$ and $N_3=N_4=\emptyset$;
  \item[(viii)] $k=4$, $|N_4|=d$, and $N_1=N_2=N_3=\emptyset$;
  \item[(ix)] $k=4$, $|N_3|=\frac{4}{3}d$ and $N_1=N_2=N_4=\emptyset$;
  \item[(x)] $k=4$, $|N_2|=2d$, and $N_1=N_3=N_4=\emptyset$.
  \end{itemize}

\smallskip
We discuss these cases in a series of lemmas at the end of the proof.
In particular, it follows from Lemmas \ref{lem:i}--\ref{lem:viii}, \ref{lem:ix} and \ref{lem:x} below that cases (i)--(x) do not occur.
This shows that Case B.1 does not happen, and we conclude that the map $\xi_{\alpha,t}\colon\{t\}\times \mathbb{P}^1\dashrightarrow C^{(s)}$ is constant for some $\alpha$ such that $N_{\alpha}\neq\emptyset$.

\medskip
\hspace{1cm}\underline{Case B.2:} let $\alpha$ be such that $N_{\alpha}=\left\{q_1,\dots,q_s\right\}\neq\emptyset$ and the map $\xi_{\alpha,t}\colon\{t\}\times \mathbb{P}^1\dashrightarrow C^{(s)}$ is constant.
In particular, $Q=q_1+\dots+q_s$ is the only point in the image of $\xi_{\alpha,t}$.
It follows that for general $y\in\mathbb{P}^1$, the points $q_1,\dots,q_s\in C$ are contained in the support of the divisor $D=D_{(t,y)}$ in \eqref{eq:divisorDbis}.
Therefore, the fiber $\varphi_t^{-1}(t,y)=\{x_1,\dots,x_d\}\subset E_t$ is such that at least one of the points $P_i=f(x_i)$ lies on the $(k-1)$-dimensional subvariety $q_1+C^{(k-1)}\subset C^{(k)}$, at least one lies on $q_2+C^{(k-1)}\subset C^{(k)}$, and so on.
Since $(t,y)$ varies on an open subset of $\{t\}\times \mathbb{P}^1$ and $E_t$ is irreducible, we deduce that 
$$f(E_t)\subset \bigcap_{j=1}^s \left(q_j+C^{(k-1)}\right),$$ 
where the points $q_j$ are distinct.

We point out that $1\leq s\leq k-1$. Indeed, if $s\geq k$, then the intersection above would have dimension smaller than 1, and it could not contain the curve $f(E_t)$.

If $s=k-1$, then $Q=q_1+\dots+q_{k-1}$ and $f(E_t)\subset \bigcap_{j=1}^s \left(q_j+C^{(k-1)}\right)= q_1+\dots+q_{k-1}+C$.
Since both $f(E_t)$ and $C_Q:=q_1+\dots+q_{k-1}+C$ are irreducible curves, they do coincide.
In particular, this concludes the proof of Theorem \ref{thm:covfam2} for the case $k=2$ as $1\leq s\leq k-1$.

Therefore, we assume $k=\{3,4\}$ and $1\leq s\leq k-2$.
Then we consider the constant map $\xi_{\alpha,t}\colon\{t\}\times \mathbb{P}^1\dashrightarrow C^{(s)}$ such that $(t,y)\longmapsto Q$, and we define the map $\phi\colon T\dashrightarrow C^{(s)}$, sending a general $t\in T$ to the unique point $Q\in C^{(s)}$ in the image of $\xi_{\alpha,t}$.
As $f(E_t)\subset Q+C^{(k-s)}$ and $\cE\stackrel{\pi}{\longrightarrow}T$ is a covering family, the map $\phi$ is dominant, with general fiber $\phi^{-1}(Q)$ of dimension $\dim T-\dim C^{(s)}=k-s-1$.
Thus the pullback 
\begin{equation}\label{eq:pullback}
\cE\times_T \phi^{-1}(Q)\longrightarrow \phi^{-1}(Q)
\end{equation}
is a $(k-s-1)$-dimensional family of $d$-gonal curves $E_t$, with $f(E_t)\subset Q+C^{(k-s)}$ and $d=\gon(C)$. 
Moreover, since $\cE\stackrel{\pi}{\longrightarrow}T$ is a covering family, we conclude that the family \eqref{eq:pullback} covers $Q+C^{({k-s})}$. 

Assuming $(k,s)=(3,1)$ or $(k,s)=(4,2)$, and using the natural identification between $Q+C^{(2)}$ and $C^{(2)}$, it follows from the assertion for $k=2$ that for general $t\in \phi^{-1}(Q)$, there exists $q'\in C$ such that $f(E_t)=Q+q'+C$, where $Q\in C^{(s)}$, i.e. $f(E_t)=C_{Q+q'}$.
In particular, this concludes the proof of Theorem \ref{thm:covfam2} for $k=3$.

Finally, if $(k,s)=(4,1)$, we have that $Q=q_1\in C$, and the pullback family is a covering family of $d$-gonal curves for $q_1+C^{(3)}$.
Then the assertion in the case $k=3$ implies that the general fiber $f(E_t)$ has the form $f(E_t)=q_1+q'+q''+C$ for some $q'+q''\in C^{(2)}$, and this concludes the proof in the remaining case $k=4$.
\end{proof}

In order to complete the proof of Theorem \ref{thm:covfam2}, we need to discuss configurations (i)--(x) listed in Case B.1 above.
To this aim, we recall the following fact (see \cite[Claim 3.6]{BP1}) and, for the sake of completeness, we include its elementary proof.
\begin{lemma}\label{lem:d-k}
Consider the points $P_1,\dots,P_d$ in \eqref{eq:P_i}. 
Let $A\subset \{q_1,\ldots,q_s\}$ be a set of points and let $j\in \{1,\dots,d\}$ such that $A\cap\Supp(P_j)=\emptyset$, and $A\cap\Supp(P_i)\neq\emptyset$ for any $i\neq j$.
Then $$|A|\geq \gon(C)-k.$$
\begin{proof}
Consider the $(k-1)$-planes $\Lambda_1,\dots,\Lambda_d$ defined in \eqref{eq:Lambda_i} and the linear space $\Span(A)\subset \mathbb{P}^{g-1}$.
Of course, if $\dim\Span(A)> g-1-k$, then $\Span(A)\cap \Lambda_j\neq\emptyset $.
If instead $\dim\Span(A)\leq g-1-k$, then $\Span(A)$ intersects $\Lambda_j$ because $\Lambda_1,\dots,\Lambda_d$ are $\SP(g-1-k)$ and $A\cap\Lambda_i\neq\emptyset$ for any $i=1,\dots,\widehat{j},\dots,d$.
In any case $\Span(A)\cap \Lambda_j\neq \emptyset$, so that $\dim \Span(A,\Lambda_j)\leq \dim \Span(A)+\dim \Lambda_j\leq |A|+k-2$.
Moreover, $\Span(A,\Lambda_j)$ contains (at least) $|A|+k$ distinct points of $C$ (the elements of $A$ and the $k$ points supporting $P_j$).
By the geometric version of the Riemann--Roch theorem, those points define a $\mathfrak{g}^r_{|A|+k}$ on $C$, with $r\geq 1$.
Thus $|A|+k\geq \gon(C)$.
\end{proof}
\end{lemma}

We use throughout the notation and the setting of Case B.1.

\begin{lemma}\label{lem:i}
Configuration (i) of Case B.1---that is $k=2$, $|N_2|=d$ and $N_1=\emptyset$---does not occur.
\begin{proof}
If $k=2$ and $|N_2|=d$, the divisor $D$ has the form $D=2(q_1+\dots+q_d)$.

Suppose that $d=3$, so that $|q_1+q_2+q_3|$ is a complete $\mathfrak{g}^1_3$ on $C$.
It follows from \eqref{eq:martens} that $C$ possesses finitely many $\mathfrak{g}^1_3$.
Hence we obtain a contradiction, because $q_1+q_2\in C^{(2)}$ is a general point, but triples $q_1+q_2+q_3$ defining a $\mathfrak{g}^1_3$ describe just a 1-dimensional locus.

Analogously, if $d=4$, then $|q_1+\dots+q_4|$ is a $\mathfrak{g}^1_4$ on $C$.
As $\gon(C)=4$ and $C$ is neither bielliptic nor a smooth plane quintic, \eqref{eq:mumford} implies that $C$ possesses finitely many $\mathfrak{g}^1_4$.
Thus we obtain a contradiction, as $q_1+q_2\in C^{(2)}$ is a general point, but 4-tuples $q_1+\dots+q_4$ defining a $\mathfrak{g}^1_4$ vary in a family of dimension 1.

Finally, suppose that $d\geq 5$. 
Given a point $x\in \{q_1,\dots,q_d\}$, it belongs to the support of two points of $C^{(2)}$, say $P_1$ and $P_2$. 
Then $\Supp(P_1)\cup\Supp(P_2)$ consists of 3 distinct points. 
As $d\geq 5$, there exists a point $y\in \{q_1,\dots,q_d\}\smallsetminus\{x\}$ belonging to the support of two other points of $C^{(2)}$, say $P_3$ and $P_4$.
Now, we can choose $n\leq d-5$ distinct points $a_1,\dots,a_n\in \{q_1,\dots,q_d\}\smallsetminus\{x,y\}$, such that the set $A:=\left\{x,y,a_1,\dots,a_n\right\}$ contains a point in the support of $P_i$ for any $i=1,\dots,d-1$ and $A\cap\Supp(P_d)=\emptyset$ (given $x$ and $y$, it suffices to choose---at most---one point in the support of any
$P_5,\dots, P_{d-1}$, avoiding the points of $\Supp(P_d)$).
Thus the set $A$ satisfies the assumption of Lemma \ref{lem:d-k}, but $|A|\leq d-3=\gon(C)-3$, a contradiction.
In conclusion, case (i) does not occur.
\end{proof}
\end{lemma}

We point out that Lemma \ref{lem:i} shows that Case B.1 does not occur for the second symmetric product of $C$. 
In particular, this fact completes the proof of Theorem \ref{thm:covfam2} for $k=2$.
Hence, in the proofs of the following lemmas discussing cases (ii)--(x), we may assume that Theorem \ref{thm:covfam2} holds for $k=2$.

\begin{lemma}\label{lem:ii}
Configuration (ii) of Case B.1---that is  $k=3$, $|N_1|=|N_2|=d$ and $N_3=\emptyset$---does not occur.
\begin{proof}
If $k=3$, $|N_1|=|N_2|=d$ and $N_3=\emptyset$, then $D=(c_1+\dots+c_d)+2(q_1+\dots+q_d)$, where $c_1,\dots,c_d,q_1,\dots,q_d\in C$ are distinct points.
Furthermore, when we proved that $\alpha|N_d|$ is a multiple of $d$, we saw that the number of points in $N_\alpha\cap\Supp(P_i)$ does not depend on $i$, so that any point $P_i$ has the form $P_i=c_i+q_{i_1}+q_{i_2}$, where $q_{i_1},q_{i_2}\in N_2=\{q_1,\dots,q_d\}$ are distinct points. 

We define the map 
$$\psi_t\colon f(E_t)\dashrightarrow C^{(2)}$$ 
sending $P_i=c_i+q_{i_1}+q_{i_2}$ to $\psi_t(P_i):=P_i-c_i=q_{i_1}+q_{i_2}$.
This map is clearly non-constant and its image is an irreducible curve on $C^{(2)}$, with 
$$\gon(\psi_t(f(E_t)))\leq \gon (E_t)=\gon(C).$$
Therefore, by varying $t\in T$, we obtain a family of irreducible curves covering $C^{(2)}$, because the curves $f(E_t)$ cover $C^{(3)}$. 
Thus $\gon(\psi_t(f(E_t)))=\gon(C)$ by \cite[Theorem 1.5]{B1}.

Moreover, the assertion of Theorem \ref{thm:covfam2} for $k=2$ implies that $\psi_t(f(E_t))=C_Q:=\left\{\left. Q+p\in C^{(2)}\right|p\in C\right\}$ for some point $Q\in C$.
In particular, for any $i=1,\dots,d$, the point $\psi_t(P_i)\in \psi_t(f(E_t))$ has the form $\psi_t(P_i)=q_{i_1}+q_{i_2}=q_{i_1}+Q$, but this is impossible because the points supporting $\psi(P_i)$ are elements of $N_2$, whereas $d\geq 3$.   
It follows that case (ii) does not occur. 
\end{proof}
\end{lemma}

\begin{lemma}\label{lem:iii}
Configuration (iii) of Case B.1---that is  $k=3$, $|N_3|=d$ and $N_1=N_2=\emptyset$---does not occur.
\begin{proof}
If $k=3$, $|N_3|=d$ and $N_1=N_2=\emptyset$, the divisor $D$ has the form $D=3(q_1+\dots+q_d)$.
In particular, each $q_j$ belongs to the support of three distinct points among $P_1,\dots,P_d\in C^{(3)}$.
We note that $d\neq 3$, otherwise we would have $P_1=P_2=P_3=q_1+q_2+q_3$.

\smallskip
If $d=4$, then $|q_1+\dots+q_4|$ is a $\mathfrak{g}^1_4$ and $C$ possesses finitely many $\mathfrak{g}^1_4$ by \eqref{eq:mumford}.
Hence we argue as in Lemma \ref{lem:i}, and we obtain a contradiction as $P_1=q_1+q_2+q_3\in C^{(3)}$ is a general point, but 4-tuples $q_1+\dots+q_4$ defining a $\mathfrak{g}^1_4$ describe just a 1-dimensional locus.

\smallskip
If $d=5$, then $|q_1+\dots+q_5|$ is a $\mathfrak{g}^1_5$ on $C$ and we argue analogously.
In particular, since $\gon(C)=5\leq \left\lfloor\frac{g+3}{2}\right\rfloor$, we deduce that $g\geq 7$. 
Hence \eqref{eq:mumford} implies that $C$ possesses at most a 1-dimensional family of $\mathfrak{g}^1_5$, so the locus described by $5$-tuples $q_1+\dots+q_5$ defining a $\mathfrak{g}^1_5$ has dimension 2.
Thus we get a contradiction as $P_1=q_1+q_2+q_3\in C^{(3)}$ varies on a threefold.

\smallskip
If $d=6$, we claim that there exist two points $x,y\in \{q_1,\dots,q_6\}$ such that $x+y\leq P_i$ for only one $i\in\{1,\dots,6\}$.
To see this fact, suppose---without loss of generality---that $q_1$ belongs to the support of $P_1,P_2,P_3$ and $\Supp(P_1+P_2+P_3)=\left\{q_1,\dots,q_r\right\}$ for some $4\leq r\leq 6$.
If there exists $j\in \{2,\dots,r\}$ such that $q_j$ lies on the support of only one point among $P_1,P_2,P_3$, then $q_1+q_j\leq P_i$ for only one $i\in\{1,2,3\}$, and we may set $(x,y)=(q_1,q_j)$.
If instead any $q_j$ belongs to the support of two points among $P_1,P_2,P_3$, then $\Supp(P_1+P_2+P_3)=\left\{q_1,\dots,q_4\right\}$, because the divisor $P:=P_1+P_2+P_3$ has degree 9, $\mult_{p_1}P=3$ and $\mult_{p_j}P\geq 2$ for any $2\leq j\leq r$.
Consider the point $q_5$ and, without loss of generality, suppose that $q_5\in \Supp(P_4)$.
As $s=d=6$, $\Supp(P_4)$ contains at least an element of $\left\{q_2,q_3,q_4\right\}$, say $q_2$. 
Since $q_2$ belongs also to the support of two points among $P_1,P_2,P_3$, we conclude that $q_2+q_5\leq P_i$ only for $i=4$, and we may set $(x,y)=(q_2,q_5)$.\\
We point out that $A:=\{x,y\}$ intersects all the sets $\Supp(P_1),\dots,\Supp(P_6)$ but one, because $x$ belongs to the support of $P_i$ and of two other points of $C^{(3)}$---say $P_2$ and $P_3$---and $y\in \Supp(P_i)$ belongs to the support of two points other than $P_2$ and $P_3$.
Hence $A$ satisfies the assumption of Lemma \ref{lem:d-k} with $|A|=2= d-4=\gon(C)-4$, a contradiction.

\smallskip
If $d=7$, we distinguish two cases.
Suppose that there exist distinct $x,x'\in\{q_1,\dots,q_7\}$ such that $\{x,x'\}$ is contained in the support of two points of $C^{(3)}$, say $P_1$ and $P_2$.
Therefore, denoting by $P_3$ the third point supported on $x$, we have that $\Supp(P_1+P_2+P_3)$ consists of at most 6 distinct elements of $\{q_1,\dots,q_7\}$.
Then there exists $y\in \{q_1,\dots,q_7\}\smallsetminus\Supp(P_1+P_2+P_3)$ such that $A=\{x,y\}$ satisfies the assumption of Lemma \ref{lem:d-k}, and we get a contradiction as $|A|=2= d-5<\gon(C)-3$.

If instead $x,x'$ as above do not exist, up to reordering indexes, the only admissible configuration for $P_1,\dots,P_7\in C^{(3)}$ is
\begin{align*}
& P_1=q_1+q_2+q_3,\,P_2=q_1+q_4+q_5,\,P_3=q_1+q_6+q_7,\,P_4=q_2+q_4+q_6,\\
& P_5=q_2+q_5+q_7,\,P_6=q_3+q_4+q_7,\,P_7=q_3+q_5+q_6.
\end{align*}
Then the set $A:=\{q_4,q_5,q_6\}$ satisfies the assumption of Lemma \ref{lem:d-k}, as $A\cap\Supp(P_i)=\emptyset$ if and only if $i=1$. 
However, $|A|=3=\gon(C)-4$, a contradiction.

\smallskip
Finally, suppose that $d\geq 8$.
Given a point $x\in \{q_1,\dots,q_d\}$, it belongs to the support of three points of $C^{(3)}$, say $P_1,P_2,P_3$. 
Then $\Supp(P_1+P_2+P_3)$ consists of at most 7 distinct elements of $\{q_1,\dots,q_d\}$.
As $d\geq 8$, we may consider a point $y\in \{q_1,\dots,q_d\}\smallsetminus\Supp(P_1+P_2+P_3)$ belonging to the support of three other points, say $P_4,P_5,P_6$.
Then we can choose $n\leq d-7$ distinct points $a_1,\dots,a_n\in \{q_1,\dots,q_d\}\smallsetminus\{x,y\}$, such that $A:=\left\{x,y,a_1,\dots,a_n\right\}$ satisfies the assumption of Lemma \ref{lem:d-k} and, being $|A|\leq d-5<\gon(C)-3$, we get a contradiction.

In conclusion, case (iii) does not occur. 
\end{proof}
\end{lemma}

\begin{lemma}\label{lem:iv}
Configuration (iv) of Case B.1---that is $k=3$, $|N_2|=\frac{3}{2}d$ and $N_1=N_3=\emptyset$---does not occur.
\begin{proof}
If $k=3$, $|N_2|=\frac{3}{2}d$ and $N_1=N_3=\emptyset$, the divisor $D$ has the form $D=2(q_1+\dots+q_s)$, with $s=\frac{3d}{2}$.
In particular, each $q_j$ belongs to the support of two distinct points among $P_1,\dots,P_d\in C^{(3)}$.
Moreover, we may set $d=2b$ and $s=3b$, for some integer $b\geq 2$.

\smallskip
Suppose that $b=2$, so that $d=4$ and $s=6$. We may assume that $q_1\in \Supp(P_1)\cap \Supp(P_2)$ and we consider the projection $\pi_1\colon \mathbb{P}^{g-1}\dashrightarrow \mathbb{P}^{g-2}$ from $q_1$. 
By Lemma \ref{lem:sottoSP}, the planes $\Gamma_3:=\pi_1(\Lambda_3)$ and $\Gamma_4:=\pi_1(\Lambda_4)$ are $\SP(g-5)$, so that $\Gamma_3=\Gamma_4\cong \mathbb{P}^{2}$.
It follows that $\Span(q_1,\Lambda_3,\Lambda_4)$ has dimension at most $3$ and contains (at least) 5 points among $q_1,\dots,q_6$, say $q_1,q_3,\dots,q_6$.
Thus $|q_1+q_3+\dots+q_6|$ is a $\mathfrak{g}^r_5$ with $r\geq 1$, and $\{q_1,q_3,\dots,q_6\}$ contains the support of two points among $P_2,P_3,P_4$.
Since $g\geq k+4= 7$, then \eqref{eq:mumford} implies that $r=1$ and that $C$ possesses at most a 1-dimensional family of complete $\mathfrak{g}^1_5$, contradicting the fact that each $P_i\in C^{(3)}$ is a general point. 

\smallskip
We now assume $b=3$, so that $d=6$ and $s=9$.
Up to reordering indexes, there exist two distinct points $x,y\in \{q_1,\dots,q_d\}$, with $x\in\Supp(P_1)\cap\Supp(P_2)$ and $y\in\Supp(P_3)\cap\Supp(P_4)$.
Let $\ell:=\Span(x,y)$, and notice that it does not intersect $\Lambda_5$ and $\Lambda_6$ (otherwise we would have a $\mathfrak{g}^1_5$, but $\gon(C)=d=6$).  
Then we consider the projection $\pi_\ell\colon \mathbb{P}^{g-1}\dashrightarrow \mathbb{P}^{g-3}$ from $\ell$, and we deduce from Lemma \ref{lem:sottoSP} that the planes $\Gamma_5:=\pi_\ell(\Lambda_5)$ and $\Gamma_6:=\pi_\ell(\Lambda_6)$ are $\SP(g-6)$.
Hence $\Gamma_5=\Gamma_6\cong \mathbb{P}^{2}$ and $\Span(x,y,\Lambda_5,\Lambda_6)$ has dimension at most $4$.

If $\Supp(P_5+P_6)$ consists of at least 5 points, then we may define a $\mathfrak{g}^r_7$ with $r\geq 2$.
By Lemma \ref{lem:very ample} we deduce that $r=2$ and the complete linear series gives an embedding.
Therefore, it is the only $\mathfrak{g}^2_7$ on $C$ by \cite[Teorema 3.14]{Cil}, and we get a contradiction, because we may consider the point $P_5$ to be general on $C^{(3)}$.
  
Otherwise $|\Supp(P_5+P_6)|=4$, and $\Supp(P_5)\cap\Supp(P_6)$ consists of exactly two points, say $q_8,q_9$.
Since the points $P_1,\dots,P_6$ are indistinguishable, then for any $i=1,\dots,6$, there exists $j\neq i$ such that $|\Supp(P_i)\cap\Supp(P_j)|=2$, i.e.
 the points $P_i$ have the form
\begin{equation*}
P_1=q_1+q_2+q_3, P_2=q_2+q_3+q_4, P_3=q_4+q_5+q_6, P_4=q_5+q_6+q_7, P_5=q_7+q_8+q_9, P_6=q_8+q_9+q_1.
\end{equation*}
Hence we may consider the projection $\pi_{\ell'}\colon \mathbb{P}^{g-1}\dashrightarrow \mathbb{P}^{g-3}$ from the line $\ell':=\Span(q_1,q_4)$, so that the planes $\Gamma'_4:=\pi_{\ell'}(\Lambda_4)$ and $\Gamma'_5:=\pi_{\ell'}(\Lambda_5)$ are $\SP(g-6)$.
It follows that $\Span(q_1,q_4,\Lambda_4,\Lambda_5)$ has dimension at most $4$ and $\Supp(P_4+P_5)$ consists of 5 points.
So $|q_1+q_4+\dots+q_9|$ is a $\mathfrak{g}^r_7$ with $r\geq 2$, and we can conclude as above.

\smallskip
Finally, suppose that $b\geq 4$.
As $s=3b\geq 12$, up to reorder indexes, we may consider three distinct points $x,y,z\in \{q_1,\dots,q_d\}$, with $x\in\Supp(P_1)\cap\Supp(P_2)$, $y\in\Supp(P_3)\cap\Supp(P_4)$ and $z\in\Supp(P_5)\cap\Supp(P_6)$.
Then we can choose $n\leq d-7$ distinct points $a_1,\dots,a_n\in \{q_1,\dots,q_d\}\smallsetminus\{x,y,z\}$, such that the set $A:=\left\{x,y,z,a_1,\dots,a_n\right\}$ contains a point in the support of $P_i$ for any $i=1,\dots,d-1$ and $A\cap\Supp(P_d)=\emptyset$.
Thus $A$ satisfies the assumption of Lemma \ref{lem:d-k}, with $|A|\leq d-4=\gon(C)-4$, a contradiction.

Therefore, we conclude that case (iv) does not occur.
\end{proof}
\end{lemma}

We note that Lemmas \ref{lem:ii}, \ref{lem:iii} and \ref{lem:iii} conclude the analysis of Case B.1 for the $3$-fold symmetric product of $C$.
In particular, the proof of Theorem \ref{thm:covfam2} for $k=3$ is now complete.
So, in the analysis of the remaining cases (v)--(x), we assume that Theorem \ref{thm:covfam2} holds for $k=3$ also.

\begin{lemma}\label{lem:v}
Configuration (v) of Case B.1---that is  $k=4$, $|N_1|=|N_3|=d$ and $N_2=N_4=\emptyset$---does not occur.
\begin{proof}
If $k=4$, $|N_1|=|N_3|=d$ and $N_2=N_4=\emptyset$, then $D=(c_1+\dots+c_d)+3(q_1+\dots+q_d)$.
Moreover, as the number of points in $N_\alpha\cap\Supp(P_i)$ does not depend on $i$, any point $P_i$ has the form $P_i=c_i+q_{i_1}+q_{i_2}+q_{i_3}$, where $q_{i_1},q_{i_2},q_{i_3}\in N_3=\{q_1,\dots,q_d\}$ are distinct points. 

If $d=3$, then $|q_1+q_2+q_3|$ is a $\mathfrak{g}^1_3$ on $C$. Since $P_1=c_1+q_1+q_2+q_3\in C^{(4)}$ is a general point, then also $q_1+q_2+q_3\in C^{(3)}$ is. Thus we get a contradiction, because divisors $q_1+q_2+q_3$ defining a $\mathfrak{g}^1_3$ vary in a 1-dimensional family by \eqref{eq:martens}. 

If $d\geq 4$, we proceed as in the proof of Lemma \ref{lem:ii}.
Namely, we define the map $\psi_t\colon f(E_t)\dashrightarrow C^{(3)}$ sending $P_i=c_i+q_{i_1}+q_{i_2}+q_{i_3}$ to $\psi_t(P_i):=P_i-c_i=q_{i_1}+q_{i_2}+q_{i_3}$.
The image of this map is an irreducible curve on $C^{(3)}$, with $\gon(\psi_t(f(E_t)))\leq \gon (E_t)=\gon(C)$.
As we vary $t\in T$, the curves $\psi_t(f(E_t))$ describe a family of irreducible curves covering $C^{(3)}$ and having gonality $\gon(\psi_t(f(E_t)))\leq \gon(C)$.
Since $\covgon(C^{(3)})= \gon(C)$, we deduce that $\gon(\psi_t(f(E_t)))=\gon(C)$.
Then the assertion of Theorem \ref{thm:covfam2} for $k=3$ implies that $\psi_t(f(E_t))=C_Q:=\left\{\left. Q+p\in C^{(3)}\right|p\in C\right\}$ for some point $Q:=y_1+y_2\in C^{(2)}$.
Therefore, for any $i=1,\dots,d$, the point $\psi_t(P_i)\in \psi_t(f(E_t))$ has the form $\psi_t(P_i)=q_{i_1}+q_{i_2}+q_{i_3}=q_{i_1}+y_1+y_2$, but this is impossible because the points supporting $\psi_t(P_i)$ are elements of $N_3$, so they can not lie in the support of all the points $P_1,\dots,P_d$ with $d\geq 4$.   

Thus case (v) does not occur.
\end{proof}
\end{lemma}

\begin{lemma}\label{lem:vi}
Configuration (vi) of Case B.1---that is  $k=4$, $|N_1|=2d$, $|N_2|=d$ and $N_3=N_4=\emptyset$---does not occur.
\begin{proof}
 If $k=4$, $|N_1|=2d$, $|N_2|=d$ and $N_3=N_4=\emptyset$, then $D=(c_1+\dots+c_{2d})+2(q_1+\dots+q_d)$ and any point $P_i$ has the form $P_i=c_{i_1}+c_{i_2}+q_{i_3}+q_{i_4}$, where the points $c_{i_1},c_{i_2}\in N_1=\{c_1,\dots,c_{2d}\}$, $q_{i_3},q_{i_4}\in N_2= \{q_1,\dots,q_d\}$ are distinct.
Given the map $\psi_t\colon f(E_t)\dashrightarrow C^{(2)}$ sending $P_i=c_{i_1}+c_{i_2}+q_{i_3}+q_{i_4}$ to $q_{i_3}+q_{i_4}$, and using the very same argument of Lemma \ref{lem:ii}, we conclude that case (vi) does not occur.
\end{proof}
\end{lemma}

\begin{lemma}\label{lem:vii}
Configuration (vii) of Case B.1---that is  $k=4$, $|N_1|=d$, $|N_2|=\frac{3}{2}d$ and $N_3=N_4=\emptyset$---does not occur.
\begin{proof}
If $k=4$, $|N_1|=d$, $|N_2|=\frac{3}{2}d$ and $N_3=N_4=\emptyset$, then $D=(c_1+\dots+c_{d})+2(q_1+\dots+q_s)$, with $s=\frac{3}{2}d$.
Moreover, any point $P_i$ has the form $P_i=c_{i}+q_{i_1}+q_{i_2}+q_{i_3}$, where $q_{i_1},q_{i_2},q_{i_3}\in N_2= \{q_1,\dots,q_s\}$ are distinct.
Analogously, by considering the map $\psi_t\colon f(E_t)\dashrightarrow C^{(3)}$ in Lemma \ref{lem:v}, and retracing the argument therein, we rule out case (vii) also. 
\end{proof}
\end{lemma}

\begin{lemma}\label{lem:viii}
Configuration (viii) of Case B.1---that is  $k=4$, $|N_4|=d$, and $N_1=N_2=N_3=\emptyset$---does not occur.
\begin{proof}
If $k=4$, $|N_4|=d$, and $N_1=N_2=N_3=\emptyset$, the divisor $D$ has the form $D=4(q_1+\dots+q_d)$, where $|q_1+\dots+q_d|$ is a $\mathfrak{g}^1_d$ on $C$.

Clearly $d\neq 4$, otherwise $P_1=P_2=P_3=P_4$. We note further that $d\neq 5, 6$, because divisors of degree $d=5,6$ defining a $\mathfrak{g}^1_d$ vary in family of dimension at most $3$ by \eqref{eq:mumford}, but $P_1=q_1+q_2+q_3+q_4$ is a general point of $C^{(4)}$.

In order to discuss the case $d\geq 7$, we prove the following.
\begin{claim}\label{claim:case_viii}
There exist two distinct points $x,y\in\{q_1,\dots,q_d\}$ such that $\{x,y\}$ intersects the support of exactly $\alpha$ points among $P_1,\dots,P_d$, with $6\leq\alpha\leq d-1$.
\begin{proof}[Proof of Claim \ref{claim:case_viii}]
Initially, we prove that $\alpha\geq 6$.
Without loss of generality, we assume that $q_1$ lies in the support of $P_1,P_2,P_3,P_4$ and $P_5=q_2+q_3+q_4+q_5$.
If there exists $q_j \in\{q_2,\dots,q_5\}$, such that $q_j\in\Supp(P_i)$ for some $i\geq 6$, then $\{x,y\}=\{q_1,q_j\}$ intersects the support of at least 6 points, as wanted.
Otherwise, for any $j\in\{2,\dots,5\}$, we have that if $p_j\in\Supp(P_i)$, then $1\leq i\leq 5$.
This implies that $\mult_{q_j}(P_1+\dots+P_5)=4$ for any $j\in\{2,\dots,5\}$. 
Therefore, up to reordering indexes, the points $P_1,\dots,P_5\in C^{(4)}$ have the form
\begin{align*}
&P_1=q_1+q_2+q_3+q_4, \quad P_2=q_1+q_2+q_3+q_5, \quad P_3=q_1+q_2+q_4+q_5,\\
&P_4=q_1+q_3+q_4+q_5, \quad P_5=q_2+q_3+q_4+q_5.
\end{align*}
Then the set $\{x,y\}=\{q_1,q_6\}$ intersects the support of at least 6 points, as wanted.

Since $x,y\in N_4$, we have $\alpha\in\{6,7,8\}$, so the bound $\alpha\leq d-1$ is trivial for any $d\geq 9$.
Thus we have to prove that we can choose $q,q'$ such that $\alpha\leq d-1$ with $d=7,8$.

Assume that $d=8$. By the first part of the proof, there exists two points $q_i,q_j$ such that $\{q_i,q_j\}$ intersects the support of exactly $\beta\geq 6$ points among $P_1,\dots,P_d$.
If $\beta\leq 7$, we take $\{x,y\}=\{q_i,q_j\}$.
If instead, $\beta=8$, we may assume $\{q_i,q_j\}=\{q_1,q_2\}$, where $q_1$ lies in the support of $P_1,P_2,P_3,P_4$ and $q_2$ lies in the support of $P_5,P_6,P_7,P_8$.
Let $q_3$ be such that $1\leq\mult_{q_3}(P_5+\dots+P_8)\leq 3$ (such a point does exist as $P_5,\dots,P_8$ are distinct points). 
If $\mult_{q_3}(P_5+\dots+P_8)\neq 1$, the pair $\{x,y\}=\{q_1,q_3\}$ intersects exactly $6\leq\alpha\leq 7$ points among $P_1,\dots,P_8$. 
If  instead $\mult_{q_3}(P_5+\dots+P_8)= 1$, then $\mult_{q_3}(P_1+\dots+P_4)=3$, and the pair $\{x,y\}=\{q_2,q_3\}$ intersects exactly $6\leq\alpha\leq 7$ points among $P_1,\dots,P_8$.

Finally, suppose that $d=7$.
We claim that there are no two distinct points $q_i,q_j\in \{q_1,\dots,q_7\}$ lying on the support of the same 4 points among $P_1,\dots,P_7$.
By contradiction, suppose that $q_1,q_2$ belong to the support of $P_1,\dots,P_4$.
Since the points $P_1,\dots,P_7$ are indistinguishable, also the support of $P_7$ must contain two points with the same property.
Denoting by $q_3,q_4$ these points, they necessarily belong to the support of $P_5,P_6,P_7$ and of a point $P_i=q_1+q_2+q_3+q_4$ with $i\in \{1,\dots,4\}$.
Thus the support of $P_i$ is given by two pairs $\{q_1,q_2\}$ and $\{q_3,q_4\}$ with this property.
Hence the same must happen for all the points $P_1,\dots,P_7$, but this is impossible as $d=7$ and we can not pair all the points $q_j$. Therefore, there are no two distinct points $q_i,q_j\in \{q_1,\dots,q_7\}$ belonging to the support of the same 4 points among $P_1,\dots,P_7$.

Now, the first part of the proof ensures that there exists two points $q_i,q_j$ such that $\{q_i,q_j\}$ intersects the support of exactly $\beta\geq 6$ points among $P_1,\dots,P_d$.
If $\beta=6$, we take $\{x,y\}=\{q_i,q_j\}$.
Otherwise, $\beta=7$ and we may assume $\{q_i,q_j\}=\{q_1,q_2\}$, where $q_1$ lies in the support of $P_1,P_2,P_3,P_4$ and $q_2$ lies in the support of $P_1,P_5,P_6,P_7$.
Let $P_1=q_1+q_2+q_3+q_4$ and let us focus on the point $q_3$.
We note that $\mult_{q_3}(P_1+\dots+P_4)\neq 4$, otherwise $q_1,q_3$ would belong to the support of the same 4 points $P_1,\dots,P_4$, a contradiction.
Analogously, if $\mult_{q_3}(P_1+\dots+P_4)=1$, then $q_2,q_3$ belong to the support of the same 4 points $P_1,P_5,P_6,P_7$, which gives again a contradiction.
Hence $\mult_{q_3}(P_1+\dots+P_4)=2$ or $3$.
In the former case, we have that $\{x,y\}=\{q_1,q_3\}$ intersects the support of exactly 6 points among $P_1,\dots,P_7$, whereas in the latter case the same holds for $\{x,y\}=\{q_2,q_3\}$.
\end{proof}
\end{claim}
In the light of the claim, we can choose $n\leq d-\alpha-1\leq d-7$ distinct points $a_1,\dots,a_n\in \{q_1,\dots,q_d\}\smallsetminus\{x,y\}$, such that the set $A:=\left\{x,y,a_1,\dots,a_n\right\}$ contains a point in the support of all the points $P_i$ but one.
Thus $A$ satisfies the assumption of Lemma \ref{lem:d-k}, with $|A|\leq d-5<\gon(C)-4$, a contradiction.
In conclusion, case (viii) does not occur.
\end{proof}
\end{lemma}

\begin{lemma}\label{lem:ix}
Configuration (ix) of Case B.1---that is  $k=4$, $|N_3|=\frac{4}{3}d$ and $N_1=N_2=N_4=\emptyset$---does not occur.
\begin{proof}
If $k=4$, $|N_3|=\frac{4}{3}d$ and $N_1=N_2=N_4=\emptyset$, then $D=3(q_1+\dots+q_s)$, where $d=3b$ and $s=4b$ for some integer $b\geq 2$ (if $b$ equaled 1, then we would have $P_1=P_2=P_3=q_1+\dots+q_4$).

\smallskip
Suppose that $b\neq 2$. 
Given a point $x\in \{q_1,\dots,q_s\}$, it belongs to the support of three points of $C^{(4)}$, say $P_1,P_2,P_3$. 
Then $\Supp(P_1)\cup\Supp(P_2)\cup\Supp(P_3)$ consists of at most 10 distinct points. 
Since $s=4b\geq 12$ and $d=3b\geq 9$, there exists a point $y\in \{q_1,\dots,q_s\}\smallsetminus\{x\}$ belonging to the support of three other points of $C^{(4)}$, say $P_4,P_5,P_6$.
Hence we can choose $n\leq d-7$ distinct points $a_1,\dots,a_n\in \{q_1,\dots,q_s\}\smallsetminus\{x,y\}$, such that the set $A:=\left\{x,y,a_1,\dots,a_n\right\}$ contains a point in the support of $P_i$ for any $i=1,\dots,d-1$  and $A\cap\Supp(P_d)=\emptyset$.
Therefore $A$ satisfies the assumption of Lemma \ref{lem:d-k}, with $|A|\leq d-5<\gon(C)-4$, a contradiction.

\smallskip
If instead $b=2$, we have $d=6$ and $s=8$. 
We claim that there exist $x,y\in \{q_1,\dots,q_8\}$ such that $x+y\leq P_i$ for only one $i\in\{1,\dots,d\}$.
To see this fact, consider the point $q_1$ and, without loss of generality, suppose that $q_1$ belongs to the support of $P_1,P_2,P_3$ and $\Supp(P_1+P_2+P_3)=\left\{q_1,\dots,q_r\right\}$ for some $5\leq r\leq 8$.
If there exists $j\in \{2,\dots,r\}$ such that $q_j$ lies on the support of only one point among $P_1,P_2,P_3$, then $q_1+q_j\leq P_i$ for only one $i\in\{1,2,3\}$, and we may set $(x,y)=(q_1,q_j)$.
If instead any $q_j$ belongs to the support of two points among $P_1,P_2,P_3$, then it is easy to check that $\Supp(P_1+P_2+P_3)=\left\{q_1,\dots,q_5\right\}$, because the divisor $P:=P_1+P_2+P_3$ has degree 12, where $\mult_{p_1}P=3$ and $\mult_{p_j}P\geq 2$ for any $j=2,\dots, r$.
Then we consider the point $q_6$ and, without loss of generality, we suppose that $q_6\in \Supp(P_4)$.
As $s=8$, $\Supp(P_4)$ contains at least an element of $\left\{q_2,\dots,q_5\right\}$, say $q_2$. 
Since $q_2$ belongs also to the support of two points among $P_1,P_2,P_3$, we conclude that $q_2+q_6\leq P_i$ only for $i=4$, and we may set $(x,y)=(q_2,q_6)$.

So, let $x,y\in \{q_1,\dots,q_8\}$ as above.
Therefore $A:=\{x,y\}\subset\{q_1,\dots,q_8\}$ intersects all the sets $\Supp(P_1),\dots,\Supp(P_6)$ but one, say all but $\Supp(P_1)$, with $P_1:=q_1+\dots+q_4$.
Let $\ell\subset \mathbb{P}^{g-1}$ be the line passing through $x,y$ and consider the $3$-planes $\Lambda_1,\dots,\Lambda_6\subset \mathbb{P}^{g-1}$ defined in \eqref{eq:Lambda_i}, which are $\SP(g-5)$.
Since $\ell$ intersects $\Lambda_2,\dots,\Lambda_6$, we have that $\ell$ meets $\Lambda_1$ also.
Thus $x,y,q_1,\dots,q_4\in \Span(\ell,\Lambda_1)$ and $3\leq \dim \Span(\ell,\Lambda_1)\leq 4$, so that $|x+y+q_1+\dots+q_4|$ is a $\mathfrak{g}^r_6$ on $C$, with $1\leq r\leq 2$.
Since $g\geq k+4= 8$, \eqref{eq:mumford} gives $\dim W^r_6\leq 4-2r$.  
Thus we get a contradiction, since the locus of divisors $x+y+q_1+\dots+q_4$ as above has dimension at most 3, but $P_1=q_1+\dots+q_4$ is a general point of $C^{(4)}$.

In conclusion, case (ix) does not occur.
\end{proof}
\end{lemma}

\begin{lemma}\label{lem:x}
Configuration (x) of Case B.1---that is  $k=4$, $|N_2|=2d$, and $N_1=N_3=N_4=\emptyset$---does not occur.
\begin{proof}
The divisor $D$ has the form $D=2(q_1+\dots+q_{2d})$ and $|q_1+\dots+q_{2d}|$ is a $\mathfrak{g}^1_{2d}$.
We may assume that  $P_1:=q_1+\dots+q_4$. 
Of course $d\neq 2$, otherwise we would have $P_1=P_2$.

Let $\Lambda_1,\dots,\Lambda_d\subset \mathbb{P}^{g-1}$ be the $3$-planes defined in \eqref{eq:Lambda_i}.
We point out that for any $i=1,\dots,d$, 
\begin{equation}\label{eq:only2}
\Lambda_i\cap \{p_1,\dots,p_{2d}\}=\Supp P_i.
\end{equation}
Otherwise we would have a complete $\mathfrak{g}^1_5$, with $\dim W^1_5\leq 1$ by \eqref{eq:mumford}. 
Since $P_i\in C^{(4)}$ is a general point, we would obtain a contradiction, as the locus of divisors defining a $\mathfrak{g}^1_5$ would have dimension at most $2$.

\smallskip
If $d=3$, then $|q_1+\dots+q_{6}|$ is a complete $\mathfrak{g}^r_6$ with $1\leq r\leq 2$, and $\eqref{eq:mumford}$ ensures that $\dim W^r_6\leq 4-2r$.
Therefore, the locus of divisors defining a $\mathfrak{g}^r_6$ has dimension at most $4-r\leq 3$, but $P_1\in C^{(4)}$ is a general point, so we get a contradiction.

\smallskip
If $d=4$, we may assume $q_1\in \Supp(P_1)\cap\Supp(P_2)$ and we consider the projection $\pi_1\colon \mathbb{P}^{g-1}\dashrightarrow \bP^{g-2}$ from $q_1$.
By Lemma \ref{lem:sottoSP}, the $3$-planes $\Gamma_3:=\pi_1(\Lambda_3)$ and  $\Gamma_4:=\pi_1(\Lambda_4)$ are in special position with respect to $(g-6)$-planes. 
Hence $\Gamma_3=\Gamma_4\cong \mathbb{P}^3$, and $\overline{\pi_1^{-1}(\Gamma_3)}=\Span(q_1,\Lambda_3,\Lambda_4)\cong \mathbb{P}^4$ contains at least $6$ points of $C$, so that we obtain a $\mathfrak{g}^1_6$.
Then we argue as in case $d=3$ with respect to the general point $P_3\in C^{(4)}$, and we obtain a contradiction. 

\smallskip
Let $d=5$, and suppose that $q_1\in \Supp(P_1)\cap\Supp(P_2)$. 
Then $\Supp(P_1)\cap\Supp(P_2)$ consists of at most $7$ distinct points.
As $s=2d= 10$, there exists a point $q_j\not\in \Supp(P_1)\cup\Supp(P_2)$, and we may set $q_j\in \Supp(P_3)\cap\Supp(P_4)$.
Since the $3$-planes $\Lambda_1,\dots,\Lambda_5$ are $\SP(g-5)$, the line $\ell:=\Span(q_1,q_j)$ intersects $\Lambda_5$.
Thus $\dim \Span(q_1,q_j,\Lambda_5)=4$, so that $|q_1+q_j+P_5|$ is a $\mathfrak{g}^1_6$, a contradiction.

\smallskip
We handle the case $d\geq 11$ by using Lemma \ref{lem:d-k}. 
Namely, given a point $x\in \{q_1,\dots,q_{2d}\}\subset C$, it belongs to the support of two points of $C^{(4)}$, say $P_1$ and $P_2$. 
Then $\Supp(P_1)\cup\Supp(P_2)$ consists of at most $7$ distinct points. 
As $s=2d\geq 22$, there exists a point $y\in \{q_1,\dots,q_{2d}\}\smallsetminus\{x\}$ belonging to the support of two other points of $C^{(4)}$, say $P_3$ and $P_4$.
By iterating this process twice, we obtain four distinct points $x,y,z,w\in \{q_1,\dots,q_{2d}\}$ such that $\{x,y,z,w\}\cap \Supp(P_i)$ for any $i=1,\dots,8$.
So we can choose $n\leq d-9$ distinct points $a_1,\dots,a_n\in \{q_1,\dots,q_{2d}\}\smallsetminus\{x,y,z,w\}$, such that the set $A:=\left\{x,y,z,w,a_1,\dots,a_n\right\}$ contains a point in the support of $P_i$ for any $i=1,\dots,d-1$ and $A\cap\Supp(P_d)=\emptyset$, but this is impossible as $|A|\leq d-5<\gon(C)-4$.

\smallskip
In order to discuss the remaining cases $6\leq d\leq 10$, we note that since the points $P_i\in C^{(4)}$ must be indistinguishable and any $q_j$ appears in the support of exactly two points $P_i$, we necessarily have one of the following configurations:
\begin{enumerate}
\item[(a)] for any $i\in\{1,\dots,d\}$, there exist $j\in\{1,\dots,d\}\backslash\{i\}$ such that $\left|\Supp(P_i)\cap\Supp(P_j)\right|=1$; 
\item[(b)] for any two distinct indexes $i,j\in\{1,\dots,d\}$, either $\Supp(P_i)\cap\Supp(P_j)=\emptyset$ or $\left|\Supp(P_i)\cap\Supp(P_j)\right|=2$.
\end{enumerate}
Indeed, suppose that case (b) does not occur.
Since the points $P_i$ are indistinguishable, this means that for any $i\in\{1,\dots,d\}$, there exists $j\neq i$ such that either $\left|\Supp(P_i)\cap\Supp(P_j)\right|=1$, as in case (a), or $|\Supp(P_i)\cap\Supp(P_j)|=3$.
In the latter case, let $q\in \Supp(P_i)\smallsetminus \Supp(P_j)$. 
Since $q\in N_2$, there exists $l\neq i$ such that $\Supp(P_i)\cap \Supp(P_l)=\{q\}$, as in configuration (a).

\medskip
Now we assume that $d=6$.
We have to study the following cases:
\begin{itemize}
\item[I.] The points $P_i$ satisfy configuration (b).
Up to reordering indexes, the only admissible configurations for the points $P_i\in C^{(4)}$ are:
\begin{equation}\label{conf_points_1}
\begin{array}{lll}
 P_1=q_1+q_2+q_3+q_4, & P_2=q_3+q_4+q_5+q_{6}, & P_3=q_5+q_6+q_{7}+q_{8},\\
P_4=q_7+q_8+q_9+q_{10}, & P_5=q_9+q_{10}+q_{11}+q_{12}, & P_6=q_{11}+q_{12}+q_{1}+q_{2}.
\end{array}
\end{equation}
or
\begin{equation}\label{conf_points_2}
\begin{array}{lll}
P_1=q_1+q_2+q_3+q_4, & P_2=q_3+q_4+q_5+q_{6}, & P_3=q_5+q_6+q_{1}+q_{2},\\
P_4=q_7+q_8+q_9+q_{10}, & P_5=q_9+q_{10}+q_{11}+q_{12}, & P_6=q_{11}+q_{12}+q_{7}+q_{8}.
\end{array}
\end{equation}

\item[II.] The points $P_i$  satisfy configuration (a) and there exist three points among $q_1,\dots,q_{12}$ in the support of two distinct points $P_i$, say 
$q_2+q_3+q_4\leq P_1,P_2$.
Since the points $P_i$ are indistinguishable, up to reordering indexes, the only admissible configuration for the points $P_i\in C^{(4)}$ is
\begin{equation}\label{conf_points_3}
\begin{array}{lll}
P_1=q_1+q_2+q_3+q_4, & P_2=q_5+q_6+q_{7}+q_{8}, &  P_5=q_9+q_{10}+q_{11}+q_{12},\\
P_4=q_{5}+q_2+q_3+q_4,  &P_5=q_9+q_6+q_7+q_8, & P_6=q_1+q_{10}+q_{11}+q_{12}.
\end{array}
\end{equation}

\item[III.] The points $P_i$  satisfy configuration (a) and for any distinct $i,j\in\{1,\dots,6\}$ either $\Supp(P_i)\cap\Supp(P_j)=\emptyset$ or $\left|\Supp(P_i)\cap\Supp(P_j)\right|=1$.

\item[IV.] The points $P_i$  satisfy configuration (a) and for any $i\in\{1,\dots,6\}$ there exists $k\in\{1,\dots,6\}\backslash\{i\}$ such that $|\Supp(P_i)\cap\Supp(P_k)|=2$.
\end{itemize}

Concerning the first three cases, the following holds.

\begin{claim}\label{touch4Pi}
In cases I,II and III there exist two distinct points $x,y\in\{q_1,\dots,q_{12}\}$ such that 
\begin{itemize}
  \item[-] $\{x,y\}\cap\Supp(P_k)\neq\emptyset$ for exactly four indexes $k\in\{1,\dots,6\}$;
  \item[-] if $i,j\in\{1,\dots,6\}$ are other two indexes (i.e. $\{x,y\}\cap\Supp(P_i)=\emptyset$ and $\{x,y\}\cap\Supp(P_j)=\emptyset$), then $\Supp(P_i)\cap\Supp(P_j)=\emptyset$.
\end{itemize}

\begin{proof} If the points $P_1,\dots,P_6$  satisfy \eqref{conf_points_1}, \eqref{conf_points_2} or \eqref{conf_points_3}, it suffices to set $x=q_1$ and $y=q_7$.
In~case III, for any $i\in\{1,\dots,6\}$, there exists a unique $j\in\{1,\dots,6\}\smallsetminus\{i\}$ such that $\Supp(P_i)\cap\Supp(P_j)=\emptyset$ and $\left|\Supp(P_i)\cap\Supp(P_k)\right|=1$ for any $k\neq i,j$.
Up to reordering indexes, we may assume $\Supp(P_1)\cap\Supp(P_2)=\Supp(P_3)\cap\Supp(P_4)=\Supp(P_5)\cap\Supp(P_6)=\emptyset$.
By taking $x\in \Supp(P_1)\cap\Supp(P_3)$ and $y\in \Supp(P_2)\cap\Supp(P_4)$, the set $\{x,y\}$ satisfies the assertion of the claim.
\end{proof}
\end{claim}

Using notation as in the claim, let us set $\ell:=\Span(x,y)$.
We note that $\ell$ does intersect neither $\Lambda_i$ nor $\Lambda_j$.
Indeed, if $\ell$ intersected $\Lambda_i$, the space $H:=\Span(\ell,\Lambda_i)\cong \bP^4$ would contain six points among $q_1,\dots,q_{12}$.
Thus $|x+y+P_i|$ would be a $\mathfrak{g}^1_{6}$ on $C$ and, by arguing as above, we would get a contradiction as $P_i$ is a general point.

Then we consider the projection $\pi\colon \mathbb{P}^{g-1}\dashrightarrow \bP^{g-3}$ from $\ell$.
By Lemma \ref{lem:sottoSP}, the $3$-planes $\Gamma_i:=\pi(\Lambda_i),\Gamma_j:=\pi(\Lambda_j)$ are in special position with respect to the $(g-7)$-planes.
Hence $\Gamma_i=\Gamma_j\cong \mathbb{P}^3$, and $H:=\overline{\pi^{-1}(\Gamma_i)}=\Span(x,y,\Lambda_i,\Lambda_j)\cong \mathbb{P}^5$.
Moreover, $H$ contains ten points of $C$ as $\Supp(P_i)\cap\Supp(P_j)=\emptyset$.
Hence the divisor $L:=x+y+P_i+P_j$ gives a $\mathfrak{g}^{4}_{10}$ on $C$.

If $|L|$ is very ample, then Castelnuovo's bound implies $g\leq 9$.
On the other hand, as $6=\gon(C)\leq \left\lfloor\frac{g+3}{2}\right\rfloor$, we conclude that $g= 9$.
Therefore, $C$ maps isomorphically to an \emph{extremal} curve of degree $10$ in $\mathbb{P}^4$.
Thus \cite[Theorem III.2.5 and Corollary III.2.6]{ACGH} ensure that $C$ possesses a $\mathfrak{g}^{1}_{3}$, but this contradicts the assumption $\gon(C)=6$.

Therefore $|L|$ is not very ample and then there exist two points $p_1,p_2\in C$, such that either $|L-p_1|$ is a  $\mathfrak{g}^{4}_{9}$ or $|L-p_1-p_2|$ is a $\mathfrak{g}^{3}_{8}$.
Since $\gon(C)=6$, both these linear series satisfy the assumption of Lemma \ref{lem:very ample}, and we obtain a contradiction.

\smallskip
It remains to analyze case IV. 
In this setting, for any $i\in\{1,\dots,6\}$, there exists a unique $k\in\{1,\dots,6\}\smallsetminus\{i\}$ such that $\left|\Supp(P_i)\cap\Supp(P_k)\right|=2$.

\begin{claim}\label{claim:caseIV}
In case IV, let $P_i,P_j, P_k$ three distinct points such that $|\Supp(P_i)\cap\Supp(P_k)|=2$ and $\Supp(P_i)\cap\Supp(P_j)\neq \emptyset$. 
Then $\Supp(P_j)\cap\Supp(P_k)= \emptyset$.

\begin{proof}
Up to reordering indexes, we may assume that $|\Supp(P_1)\cap\Supp(P_2)|=|\Supp(P_3)\cap\Supp(P_4)|=|\Supp(P_5)\cap\Supp(P_6)|=2$.
Suppose by contradiction that $\Supp (P_3)$ intersects both $\Supp (P_1)$ and $\Supp (P_2)$.
Then we may set $P_1=q_1+q_2+q_3+q_4$, $P_2=q_1+q_2+q_5+q_6$ and $P_3=q_3+q_5+q_7+q_8$, where $q_7,q_8\in \Supp (P_4)$.

Since the points $P_i$ are indistinguishable, there exists a point $P_t$, whose support intersects both $\Supp (P_3)$ and $\Supp (P_4)$.
Looking at $\Supp (P_3)$, we see that $t=1$ or $2$. 
We may assume $t=1$, so that $q_4\in P_4$.

Finally,  there exists a point $P_h$, with $h=1,\dots,4$, whose support intersects both $\Supp (P_5)$ and $\Supp (P_6)$ at two distinct points.
However, this is impossible because by construction each set $\Supp (P_h)$ may share at most one element with $\Supp (P_5)\cup\Supp (P_6)$.
\end{proof}
\end{claim}

Now, let $P_1=q_1+q_2+q_3+q_4$, with $q_1,q_2\in \Supp(P_2)$, $q_3\in\Supp(P_3)$ and $q_4\in\Supp(P_4)$.
By Claim \ref{claim:caseIV} we have $\Supp(P_2)\cap\Supp(P_3)=\emptyset$, $\Supp(P_2)\cap\Supp(P_4)=\emptyset$ and $|\Supp(P_3)\cap\Supp(P_4)|\neq 2$.
So let $P_5$ and $P_6$ such that  $|\Supp(P_3)\cap\Supp(P_5)|= 2$ and  $|\Supp(P_4)\cap\Supp(P_6)|=2$.
We necessarily have $\Supp(P_2)\cap\Supp(P_5)\neq\emptyset$ and $\Supp(P_2)\cap\Supp(P_6)\neq\emptyset$, so we can suppose $P_2=q_1+q_2+q_5+q_6$, $q_5\in\Supp(P_5)$ and $q_6\in\Supp(P_6)$.

If $\Supp(P_3)\cap\Supp(P_4)\neq \emptyset$, say $q_7\in\Supp(P_3)\cap\Supp(P_4)$, then the set $\{x,y\}=\{q_6,q_7\}$ satisfies the assertion of Claim \ref{touch4Pi} (note that $\Supp(P_1)\cap\Supp(P_5)=\emptyset$). Hence we can argue as for cases I--III above, and we obtain a contradiction.

If instead $\Supp(P_3)\cap\Supp(P_4)= \emptyset$, we can suppose $\Supp(P_3)\cap\Supp(P_6)=\{q_7\}$ and $\Supp(P_4)\cap\Supp(P_5)=\{q_8\}$.
Let $\ell_1:=\Span(q_1,q_7)$ and $\ell_2:=\Span(q_2,q_7)$.
Both $\{q_1,q_7\}$ and $\{q_2,q_7\}$ have empty intersection only with the support of $P_4$ and $P_5$.
Therefore, by considering the projections from $\ell_1$ and $\ell_2$, and arguing as in cases I--III above, we obtain two linear spaces $H_1:=\Span(q_1,q_7,\Lambda_4,\Lambda_5)\cong \mathbb{P}^5$ and $H_2:=\Span(q_2,q_7,\Lambda_4,\Lambda_5)\cong \mathbb{P}^5$.
Since $|\Supp(P_4)\cup\Supp(P_5)|=7$, each of the spaces $H_1$ and $H_2$ contains at least $9$ points among $q_1,\dots,q_{12}$, $8$ of which---i.e. the point $q_7$ and the elements of $\Supp(P_4)\cup\Supp(P_5)$---lie in $H_1\cap H_2$.
Therefore, if $H_1=H_2$, then $H_1\cong \mathbb{P}^5$ contains $10$ points of $C$ and we obtain a $\mathfrak{g}^{4}_{10}$, but we saw that $C$ does not admit such linear series.
If instead $H_1\neq H_2$, then $H_1\cap H_2\cong \bP^{4}$ contains $8$ points of $C$, which define a  $\mathfrak{g}^{3}_{8}$ on $C$.
Since $\gon(C)=6$, this linear series satisfies the assumption of Lemma \ref{lem:very ample} and we get a contradiction.
This concludes the analysis of case $d=6$.

\medskip
Now we assume $d\in\{7,8,9,10\}$ and we set $L:=q_1+\dots+q_{2d}$.

\begin{claim}\label{claim:span_cases_less10} If $d\in\{7,8,9,10\}$, then $\dim\Span(\Lambda_1,\dots,\Lambda_d)=d$. \end{claim}

\begin{proof}
Since  $\deg L=2d$ and $d=\gon(C)\leq \left\lfloor\frac{g+3}{2}\right\rfloor$, we deduce $\dim |L|\leq d-1$ by Clifford's theorem.
Then the geometric version of the Riemann--Roch theorem yields 
$$\dim\Span(\Lambda_1,\dots,\Lambda_d)=\dim\Span(q_1+\dots+2d)=2d-1-\dim |L|\geq d.$$ 
In order to show that $\dim\Span(\Lambda_1,\dots,\Lambda_d)\leq d$, we assume $q_1\in\Supp(P_1)\cap\Supp(P_2)$ and we consider the projection $\pi_1\colon \mathbb{P}^{g-1}\dashrightarrow \bP^{g-2}$ from $q_1$.
Thanks to \eqref{eq:only2} we have that $p_1\not\in \Lambda_j$ for any $j=3,\dots,d$. 
Hence the $3$-planes $\Gamma_3:=\pi_1(\Lambda_3),\dots,\Gamma_d:=\pi_1(\Lambda_d)$ are in special position with respect to $(g-6)$-planes by Lemma \ref{lem:sottoSP}.

\smallskip
Suppose that the sequence $\Gamma_3,\dots,\Gamma_d$ is indecomposable in the sense of Definition \ref{def:partition}, then $\dim\Span(\Gamma_3,\dots,\Gamma_d)\leq{d-1}$ by Theorem \ref{thm:dimSP}, and the linear space $H_1:=\Span(q_1,\Lambda_3,\dots,\Lambda_d)$ satisfies $\dim H_1\leq d$.

If $P_1,\dots,P_d$ satisfy configuration (a), we can suppose $\Supp(P_1)\cap\Supp(P_2)=\{q_1\}$, so that $\Supp(P_3)\cup\dots\cup\Supp(P_d)=\{q_2,\dots,q_{2d}\}$. 
Thus $H_1=\Span(\Lambda_1,\dots,\Lambda_d)$, and the assertion follows.

If instead $P_1,\dots,P_d$ satisfy configuration (b), we can assume $\Supp(P_1)\cap\Supp(P_2)=\{q_1,q_2\}$. 
Therefore, $\Supp(P_3)\cup\dots\cup\Supp(P_d)=\{q_3,\dots,q_{2d}\}$ and $H_1$ contains $\{q_1,q_3,\dots,q_{2d}\}$.
We claim that $q_2\in H_1$ too, so that $H_1=\Span(\Lambda_1,\dots,\Lambda_d)$, and the assertion follows.
To see this fact, we assume by contradiction that $q_2\not\in H_1$, and we consider the projection $\pi_2\colon \mathbb{P}^{g-1}\dashrightarrow \bP^{g-2}$ from $q_2$.
By arguing as above, we deduce that the linear space $H_2:=\Span(q_2,\Lambda_3,\dots,\Lambda_d)$ has dimension $\dim H_2\leq d$ and contains $\{q_2,\dots,q_d\}$.
Hence $H_1\cap H_2$ has dimension at most $d-1$ and contains the points $q_3,\dots,q_{2d}$.
Then the Riemann--Roch and Clifford's theorems imply that $|q_3+\dots+q_{2d}|$ is a $\mathfrak{g}^{r}_{2d-3}$, with $r\geq d-2$.
Since $\gon(C)=d$, this linear series satisfies the assumption of Lemma \ref{lem:very ample}, and we get a contradiction as $r\geq d-2\geq 5> 2$.
Therefore $q_2\in H_1$, as claimed.

\smallskip
Finally, we suppose that the sequence $\Gamma_3,\dots,\Gamma_d$ is decomposable in the sense of Definition \ref{def:partition}, and we show that this situation does not occur.

Suppose that there exists a part $\{\Gamma_3,\Gamma_4\}$ consisting of $3$-planes satisfying $\SP(g-6)$, so that $\Gamma_3=\Gamma_4\cong \mathbb{P}^3$. 
We note further that this is always the case when $d=7$, because the number of $\Gamma_i$ is $d-2=5$.
As $\left|\Supp(P_3)\cup\Supp(P_4)\right|\geq 5$, the space $H:=\Span(q_1,\Lambda_3,\Lambda_4)\cong\bP^4$ contains at least six points of $C$ and has dimension 4.
This gives a $\mathfrak{g}^{1}_{6}$ on $C$, which does not exist as $\gon(C)=d\geq 7$.

So we assume that $d\geq 8$ and we suppose that there exists a part $\{\Gamma_3,\Gamma_4,\Gamma_5\}$ consisting of $3$-planes that satisfy $\SP(g-6)$, which is always the case when $d=8,9$.
By Theorem \ref{thm:dimSP}, their linear span has dimension at most $4$.
Moreover, $\left|\Supp(P_3)\cup\Supp(P_4)\cup\Supp(P_5)\right|\geq 6$ since $q_j\in N_2$ for any $j=1,\dots, 2d$.
Thus $H:=\Span(q_1,\Lambda_3,\Lambda_4,\Lambda_5)$ has dimension at most 5 and contains 7 points of $C$.
Then $C$ should admit a $\mathfrak{g}^{1}_{7}$, contradicting the assumption $\gon(C)=d\geq 8$.

To conclude, assume $d=10$ and that there are no parts of cardinality 2 and 3.
Then there exists a part $\{\Gamma_3,\dots, \Gamma_6\}$ consisting of $3$-planes that satisfy $\SP(g-6)$.
By arguing as above, $H:=\Span(q_1,\Lambda_3,\dots,\Lambda_6)$ has dimension at most $6$ and $\left|\Supp(P_3)\cup\dots\cup\Supp(P_5)\right|\geq 8$. 
Therefore we obtain a $\mathfrak{g}^{r}_{9}$ on $C$, with $r\geq 2$, which is impossible as $d=10$.
\end{proof}

By Claim \ref{claim:span_cases_less10} and the geometric version of the Riemann--Roch theorem, we deduce that $|L|$ is a $\mathfrak{g}^{d-1}_{2d}$ on $C$.
In order to conclude case (x), we analyze separately the following three cases: (x.a) $L$ is very ample, (x.b) there exists $p\in C$ such that $\dim|L-p|=d-1$, (x.c) there exist $p,q\in C$ such that $\dim|L-p|=\dim|L-p-q|=d-2$.

\smallskip
\hspace{1cm}(x.a) 
Assume that $|L|$ is very ample.
Then Castelnuovo's bound implies $g\leq d+4$.
On the other hand, as $d=\gon(C)\leq \left\lfloor\frac{g+3}{2}\right\rfloor$, we have $2d-3\leq g\leq d+4$, which is impossible for $d\geq 8$.
For $d=7$, we necessarily have $g=11$, and $|L|$ embeds $C$ in $\mathbb{P}^6$ as an \emph{extremal} curve of degree $14$.
Thus \cite[Theorem III.2.5 and Corollary III.2.6]{ACGH} ensure that $C$ possesses a $\mathfrak{g}^{1}_{4}$, but this contradicts the assumption $\gon(C)=d=7$. 

\smallskip
\hspace{1cm}(x.b) 
In this case there exists $p\in C$ such that $|L-p|$ is a $\mathfrak{g}^{d-1}_{2d-1}$.
If $|L-p|$ is very ample, then Castelnuovo's bound and $d\leq \left\lfloor\frac{g+3}{2}\right\rfloor$ yield $2d-3\leq g\leq d+2$, which is impossible for any $d\geq 6$.
Therefore $|L-p|$ is not very ample and then there exist two points $q_1,q_2\in C$, such that either $|L-p-q_1|$ is a  $\mathfrak{g}^{d-1}_{2d-2}$ or $|L-p_1-q_1-q_2|$ is a $\mathfrak{g}^{d-2}_{2d-3}$.
Since $\gon(C)=d$, both these linear series satisfy the assumption of Lemma \ref{lem:very ample}.
Hence we obtain a contradiction as their dimension is larger than $2$.

\smallskip
\hspace{1cm}(x.c) 
Assume that there exist $p,q\in C$ such that $|L-p-q|$ is a $\mathfrak{g}^{d-2}_{2d-2}$.
If $|L-p-q|$ is very ample, we argue as above, and we obtain $2d-3\leq g\leq d+3$, which is impossible for any $d\geq 7$.
Therefore $|L-p-q|$ is not very ample and there exist two points $w_1,w_2\in C$, such that either $|L-p-q-w_1|$ is a  $\mathfrak{g}^{d-2}_{2d-3}$ or $|L-p-q-w_1-w_2|$ is a $\mathfrak{g}^{d-3}_{2d-4}$.
Since $\gon(C)=d$, both these linear series satisfy the assumption of Lemma \ref{lem:very ample}, and we get a contradiction as their dimension is larger than $2$.

In conclusion, case (x) is not possible. 
\end{proof}
\end{lemma}

\begin{remark}\label{rem:plane quintic}
We point out that the hypothesis of Theorem \ref{thm:covfam} can be slightly weakened, replacing the assumption ``$C$ is not a smooth plane quintic'' by ``$C$ is not a smooth plane quintic with an automorphism of order 2''. 
Indeed, we used the fact that $C$ is not a smooth plane quintic only in the proof of Lemma \ref{lem:i}, in order to rule out the case $d=4$.
Therefore, it is enough to show that if $k=2$ and $d=4$, then $C$ admits an involution $\iota\colon C\longrightarrow C$.
In this setting, the divisor in \eqref{eq:divisorDbis} is $D=2(q_1+\dots+q_4)$ and $|q_1+\dots+q_4|$ is a $\mathfrak{g}^1_4$ on $C$.
Up to reordering indexes, the points $P_1,\dots,P_4\in C^{(2)}$ in \eqref{eq:P_i} are given by
\begin{equation*}
P_1=q_1+q_2, \quad P_2=q_2+q_3, \quad P_3=q_3+q_4,\quad P_4=q_4+q_1.
\end{equation*}
It follows that there is an involution on $\{q_1,\ldots,q_4\}$ not depending on the ordering, such that $q_1\mapsto q_3$ and $q_2\mapsto q_4$ (i.e. each $q$ is associated to the unique $q'$ such that $q+q'$ is not one of the points $P_i$).
Since $|q_1+\dots+q_4|$ is a base-point-free $\mathfrak{g}^1_4$, for any $j=1,\dots,4$ we have that the only effective divisor $F\in |q_1+\dots+q_4|$ such that $q_j\in \Supp(F)$ is $F=q_1+\dots+q_4$. 
Therefore, as we vary $(t,y)\in \{t\}\times \mathbb{P}^1$, the involution on $\{q_1,\ldots,q_4\}$ defines an involution $\iota\colon C\longrightarrow C$.

Thus all the curves that we need to exclude in Theorem \ref{thm:covfam}---i.e. hyperelliptic curves, bielliptic curves and plane quintics with an involution---possess some non trivial automorphism.
As a consequence, we obtain an improvement of \cite[Theorem 1.5]{B1}, which asserts that when $k=2$, if $C$ is a curve of genus $g\geq 6$ without non-trivial automorphisms and $\cE\stackrel{\pi}{\longrightarrow}T$ is a family of irreducible curves as in Theorem \ref{thm:covfam}, then the general member $E_t=\pi^{-1}(t)$ is isomorphic to $C$. 
\end{remark}

\begin{proof}[Proof of Corollary \ref{cor:conngon}]
Let $X=C^{(k)}$ and let $\Hilb(X)$ denote the Hilbert scheme of curves on $C^{(k)}$.
Let $\cE\stackrel{\pi}{\longrightarrow}T\subset \Hilb(X)$ be a family of irreducible curves on $C^{(k)}$ computing the connecting gonality of $C^{(k)}$.
In particular, given two general points $x,y\in C^{(k)}$ there exist $t\in T$ such that $x,y\in E_t=\pi^{-1}(t)$ and $\gon(E_t)=\conngon(X)$.
Therefore, the natural map $\cE\longrightarrow X$ (which includes any curve $E_t$ in $X$) defines a dominant map $\cE\times_T\cE\longrightarrow X\times X$.
Thus $\dim (\cE\times_T\cE)\geq\dim (X\times X)$ and, since $\dim (\cE\times_T\cE)=\dim T+2$, we conclude that $\dim T\geq 2k-2$.

By \cite[Theorem 1.1]{BP1}, we deduce that $\gon(E_t)=\conngon(X)\geq \covgon(X)=\gon(C)$, and Theorem \ref{thm:covfam} ensures that equality holds if and only if 
$\cE\stackrel{\pi}{\longrightarrow}T$ is the $(k-1)$-dimensional family parameterizing curves of the form $C_Q=Q+C$, with $Q\in C^{(k-1)}$.
However, as $k-1<2k-2$ for any $k\geq 2$, we conclude that $\gon(E_t)=\conngon(X)> \covgon(X)=\gon(C)$.
\end{proof}


\end{document}